\documentclass[a4paper,12pt]{article}
\usepackage{amsmath,amssymb,amsthm} %Some extra symbols
\usepackage{mathtools}
\usepackage{float} 
\usepackage{color}
\usepackage{graphics}
\usepackage{enumerate}
\usepackage[dvipsnames]{xcolor}
\usepackage{tikz}
\usetikzlibrary{decorations.markings}
\usetikzlibrary{patterns}

\usepackage[normalem]{ulem}   %\sout{Texte à barrer}   \xout{Texte à hachurer}   \uwave{Texte à souligner par une vaguelette}

\pgfdeclarelayer{bg}    % declare background layer
\pgfsetlayers{bg,main} 

\numberwithin{equation}{section}

\newtheorem{theoremINTRO}{Theorem}

\newtheorem{propositionINTRO}[theoremINTRO]{Proposition}
\newtheorem{corollaryINTRO}[theoremINTRO]{Corollary}
\theoremstyle{definition}
\newtheorem{definitionINTRO}[theoremINTRO]{Definition}
\newtheorem{remarkINTRO}[theoremINTRO]{Remark}

\theoremstyle{theorem}
\newtheorem{theorem}{Theorem}[section]
\newtheorem{lemma}[theorem]{Lemma}
\newtheorem{proposition}[theorem]{Proposition}
\newtheorem{corollary}[theorem]{Corollary}
\theoremstyle{definition}

\newtheorem{remark}[theorem]{Remark}

\def\les{\lesssim}

\newcommand{\bks}{\backslash}
\newcommand{\R}{\mathbf{R}}

\newcommand{\Z}{\mathbf{Z}}
\newcommand{\T}{\mathbf{T}}

\newcommand{\hel} {
\hskip2.5pt{\vrule height7pt width.5pt depth0pt}
\hskip-.2pt\vbox{\hrule height.5pt width7pt depth0pt}
\, }
\newcommand{\restr}{\hel}
\newcommand{\mc}{\mathcal}

\newcommand{\sm}{\backslash}
\newcommand{\pt}{\partial}
\newcommand{\vhi}{\varphi}
\newcommand{\eps}{\varepsilon}
\newcommand{\te}{\theta}
\newcommand{\oo}{\infty}
\newcommand{\nb}{\nabla}
\newcommand{\ov}{\overline}
\newcommand{\h}{\mathcal{H}}
\newcommand{\dw}{\downarrow}
\newcommand{\up}{\uparrow}

\newcommand{\longto}{\longrightarrow}
\newcommand{\Id}{\text{Id}}
\newcommand\un{\bs{1}}
\newcommand{\Om}{\Omega}
\newcommand{\om}{\omega}
\newcommand{\ds}{\displaystyle}

\newcommand{\be}{\begin{equation}}

\newcommand{\ee}{\end{equation}}

\newcommand{\U}{S}

\def\XXint#1#2#3{{\setbox0=\hbox{$#1{#2#3}{\int}$} 
      \vcenter{\hbox{$#2#3$}}\kern-.5\wd0}}

\newcommand{\mint}{\diagup \!\!\!\!\!\!\!\ds\int}
\newcommand{\smint}{\mbox{{\scriptsize$\diagup$}}\hspace{-9.2pt}\int}

\def\lt{\left}
\def\rt{\right}
\def\bs{\boldsymbol}

\def\E{\mathcal{E}}
\def\F{\mathcal{E}}
\def\G{\mathcal{G}}

\DeclareMathOperator{\Lip}{Lip}

\DeclareMathOperator{\SO}{SO}
\DeclareMathOperator{\GL}{GL}
\DeclareMathOperator{\supp}{supp}
\DeclareMathOperator{\osc}{osc}

\DeclareMathOperator{\Span}{span}

\title{Non-convex functionals penalizing {simultaneous} oscillations along  independent directions: rigidity estimates}

\author{M. Goldman\footnote{ Universit\'e Paris-Diderot, Sorbonne Paris-Cit\'e, Sorbonne Universit\'e,  CNRS,  Laboratoire Jacques-Louis Lions, LJLL, F-75013 Paris, email: goldman@math.univ-paris-diderot.fr} \and B. Merlet\footnote{Univ. Lille, CNRS, UMR 8524, Inria - Laboratoire Paul Painlev\'e, F-59000 Lille, email: benoit.merlet@univ-lille.fr}}
\begin{document}
\maketitle

\begin{abstract}
We study a family of non-convex functionals $\{\F\}$ on the space of measurable functions $u:\Om_1\times\Om_2\subset\R^{n_1}\times\R^{n_2}\to\R$.
These functionals vanish on the non-convex subset $\U(\Om_1\times\Om_2)$ formed by  functions of the form $u(x_1,x_2)=u_1(x_1)$ or $u(x_1,x_2)=u_2(x_2)$. We investigate under which conditions the { converse} implication ``$\F(u)=0 \Rightarrow u\in \U(\Om_1\times\Om_2)$'' holds.
In particular, we show that the answer depends strongly on the smoothness of $u$. We also obtain quantitative versions of this implication by proving  
 that (at least for some parameters) $\F(u)$ controls in a strong sense the distance of $u$ to  $\U(\Om_1\times\Om_2)$.
\end{abstract}
%\tableofcontents

 \section{Introduction}\label{SIntro}
Given two bounded non-empty connected {and}  open  sets $\Om_1\subset \R^{n_1}$, $\Om_2\subset\R^{n_2}$, we consider the \textit{non-convex} set $\U(\Om)$ of measurable functions $u$ defined on $\Om=\Om_1\times\Om_2\subset\R^{n_1+n_2}$ which only depend on the first $n_1$ coordinates or on the last $n_2$ coordinates,
that is $u(x)=u_1(x_1)$ or $u(x)=u_2(x_2)$. We present and study a family of non-convex  functionals that detect whether a measurable function $u$ defined on $\Om$ belongs to $\U(\Om)$. As a first guess, we could expect that for $u$ with Sobolev regularity the relation 
\be\label{ProdSobolev}
|\nb_1u|\,|\nb_2 u|\equiv0\ \text{ a.e. in }\Om
\ee  would imply $u\in \U(\Om)$. As we will see later on, this is not the case even for $u\in Lip(\Om)$ and therefore \eqref{ProdSobolev} cannot be the starting point  for defining our functionals. However, our construction relies on a discrete version of~\eqref{ProdSobolev}.  If $u\in \U(\Om)$ then at any point $x=(x_1,x_2)\in \Om$, and any $z=(z_1,z_2)\in\R^{n_1+n_2}$ { with $|z|$} small enough, the product  
\be\label{prod}
|u(x_1+z_1,x_2)-u(x)|\, |u(x_1,x_2+z_2)-u(x)|
\ee
is well defined and vanishes. The functionals that we consider are based on the integration of~\eqref{prod} in $x$ and $z$ with a weight depending on $z$ and a limiting process that localizes the integration around $z=0$.
 
\subsection{Definitions and context}
Let us introduce some notation.  For convenience, we write the Cartesian product as a sum: we decompose the $n$-dimensional Euclidean space $X=\R^n$ as,
\[
\R^n=X_1\oplus X_2,\quad\text{with}\quad n_l:=\dim X_l\geq 1 \text{ for }l=1,2.
\]
We will assume for simplicity { that} $X_1\perp X_2$. The domain $\Om\subset\R^n$ is $\Om:=\Om_1+\Om_2$ where, for $l=1,2$, $\Om_l$ is a non-empty bounded domain (i.e. open and connected) of the subspace $X_l$. We note $L(U)$ the space of Lebesgue-measurable functions defined on a measurable set $U$ and for $x\in \R^n$, we note $x=x_1+x_2$ its decomposition in $X_1+X_2$. With this notation we define the set
\[
\U(\Om):=\{u\in L(\Om) : u(x)=u_l(x_l) \text{ in }\Om
\text{ with }l=1 \text{ or } l=2 \text{ and } u_l\in L(\Om_l)\}.
\]
(This set depends on $\Om_1$, $\Om_2$ but as no  ambiguity arises we choose this short notation.)

%\begin{multline*}
%\U(\Om):=\\
%	\lt\{ u:\Om\to\R\text{ such that }u(x)=u_1(x_1) \text{ in }\Om 
%	\text{ for some measurable function } u_1:\Om_1\to\R\rt\}\\
%	\bigcup 
%	\lt\{ u:\Om\to\R\text{ such that }u(x)=u_2(x_2) \text{ in }\Om 
%	\text{ for some measurable function }u_2:\Om_2\to\R \rt\}
%\end{multline*}
% and more generally, we use these functionals to measure the distance of $u$ to the set $\U(\Om)$. 
\noindent
We also fix a radial non-negative kernel 
\[
\rho\in L^1(\R^n,\R_+)\quad \text{ with }\int \rho=1,\text{ and }\ \supp\rho\subset B_1.
\]
As usual, for $\eps>0$ we introduce the rescaled kernel $\rho_\eps:=\eps^{-n}\rho(\eps^{-1}\cdot)$ so that $\{\rho_\eps\}_{\eps>0}$ forms a family of radial mollifiers. We introduce three real parameters $p$, $\te_1$, $\te_2>0$ and define for any measurable function $u:\Om\to \R$ and any $\eps>0$, the quantity
\be\label{Eeps}
\E_{\eps,p}^{\te_1,\te_2}(u;\Om)=\E_{\eps,p}^{\te_1,\te_2}(u):=\int_{\R^n}\rho_\eps(z) \int_{\Omega^\eps} \dfrac{|u(x+z_1)-u(x)|^{\te_1} \, |u(x+z_2)-u(x)|^{\te_2}}{|z|^p} \,dx \,dz,
\ee
with
\[
\Om^\eps:=\Om^\eps_1+\Om^\eps_2, \qquad\text{and } \qquad \Om^\eps_j:=\!\!\!\!\bigcap\limits_{z_l\in X_l,|z_l|<\eps} (\Om_l-z_l)\quad \text{ for }l=1,2.
\]
Eventually, we send $\eps$ to $0$ and define the functional
 \[
\F_{p}^{\te_1,\te_2}(u;\Om)=\F_{p}^{\te_1,\te_2}(u):=\liminf_{\eps\dw 0} \E_{\eps,p}^{\te_1,\te_2}(u).
 \]
Most of the time, we omit  the dependency on the parameters $\te_1,\te_2$ and note 
\[
\E_{\eps,p}(u):=\E_{\eps,p}^{\te_1,\te_2}(u),\quad \F_{p}(u):=\F_{p}^{\te_1,\te_2}(u).
\]
We are interested in  the qualitative and quantitative properties of functions with  finite energy $\F_p(u)$. We first observe that $u$, $\te_1$ and $\te_2$ being fixed, there exists at most one value of $p$ for which $0<\F_{p}(u)<\oo$. 
\begin{remarkINTRO}\label{remark1}
From the properties of the kernel $\rho$, we have for $p<q$, 
\[
\E_{\eps,p}(u)\leq \eps^{q-p} \,\E_{\eps,q}(u).
\]
Sending $\eps$ to $0$, we see that $
[\F_{p}(u)>0] \Rightarrow [\F_{q}(u)=\oo]$ and we deduce that there exists $ p^*(u)\in[0,\infty]$ such that 
\[
\F_{p}(u)= \begin{cases}
 0 & \text{ for $p<p^*(u)$,}\\
\infty & \text{ for $p>p^*(u)$.}
\end{cases}
\] 
\end{remarkINTRO}

By construction, $\E_{\eps,p}(u)=\F_{p}(u)=0$ for every $u\in \U(\Om)$.
% We are interested in some converse properties, more precisely, we study the qualitative and quantitative properties of measurable functions $u$ such that 
% \[
% \F( u):= \F_{p}(u):=\liminf_{\eps\dw 0} \E_{\eps,p}(u)\, < \oo. 
% \]
A first natural question  is: \be\label{question}
\text{``Does $ \F_{p}(u)=0$ implies $u\in \U(\Om)$?''}
\ee
We will see that the answer is ``yes'' for $p$ large enough depending on $\te_1,\te_2$.  This question initially appeared (with $u$ being a characteristic function
and $n_2=1$)  in the study of  pattern formation in some variational models involving competition between a local attractive term and a non-local repulsive one. 
Indeed, in~\cite{GRun,DanRun}), energies related to $\F_{p}$ are used to show that some sets $S \subset \R^n$ are union of  stripes. The functionals $\F_{p}$ extend this setting to general functions and to general dimensions $(n_1, n_2)$.\medskip

Our second main result  may be seen as an answer to a  quantitative version of Question~\eqref{question}. Indeed, we prove (at least for some values of $\theta_1$ and $\theta_2$), that the non-convex energy $\F_{p}(u)$ controls the
distance from $u$ to  the non-convex set $\U(\Om)$ in a strong norm. Of course, as seen from Remark~\ref{remark1}, this is an interesting question only for the borderline exponent $p$ for Question~\eqref{question}. \medskip
\begin{remarkINTRO}\label{remarkLinfty}
Let us point out that when investigating Question~\eqref{question}, there is no loss of generality in assuming that $u\in L^\infty$. Indeed, $\arctan u\in \U(\Om)$ if and only if $u\in \U(\Om)$ and $\F_p(\arctan u)\le \F_p(u)$.
\end{remarkINTRO}
\begin{remarkINTRO}
Notice that the functionals $ \F_{p}^{\te_1,\te_2}$ can be seen as variants  of the non-local functionals used by Brezis \textit{et al} to characterize 
Sobolev spaces~\cite{BBM2001,Brezis2002,MMS}. It turns out that the present \emph{non-convex} setting is rather different but we do use their results in our analyses at some point: when $u(x)$ splits as $u_1(x_1)+u_2(x_2)$.
\end{remarkINTRO}

\subsection{The vanishing energy case: Question~\eqref{question}}
To get some insight into the behavior of the functional $\F_{p}$, let us first consider the simple situation $u\in C^1(\ov \Om)$. Within this setting, the problem is rigid as soon as $p\ge \theta_1+\theta_2$.
\begin{propositionINTRO}[Proposition~\ref{prop:C1}]\label{C1:intro} Let us note $\te:=\te_1+\te_2$. 

\begin{itemize}
 \item[(i)] For every $u\in \Lip(\Om)$, $\F_{\te}(u)<\oo$ (and therefore $\F_{p}(u)=0$ for $p<\te$);
 \item[(ii)] If $u\in C^1(\ov {\Om})$, then\ \ $[\F_{\te}(u)=0] \Longrightarrow u\in \U(\Om)$. 
\end{itemize}
\end{propositionINTRO}
To obtain the point~(i) we simply plug the inequality $|u(x+z_l)-u(x)|\leq \|\nb u\|_\oo |z_l|$ in the definition of $\E_{\eps,\te}$. This first point shows that the parameter $p=\te$ is sharp in~(ii). The proof of~(ii) runs as follows. Using the relation $u(x+z_l)-u(x)=\nb u(x)\cdot z_l +o(|z|)$ in the definition of $\E_{\eps,\te}(u)$, we obtain that $\F_{\te}(u)=0$ implies that $\nb u$ satisfies the differential inclusion (equivalent to~\eqref{ProdSobolev})
\begin{equation}\label{diffinclu}
\nb u(x)\in X_1\cup X_2\quad\text{almost everywhere in } \Om.
\end{equation}
This differential inclusion is rigid for $u\in C^1(\Om)$: it yields $u\in \U(\Om)$. On the contrary,  since the convex hull of $X_1\cup X_2$ is $\R^n$, the differential inclusion~\eqref{diffinclu} is not rigid in the class of Lipschitz continuous functions.
Indeed, the set of Lipschitz functions satisfying~\eqref{diffinclu} is dense in $W^{1,1}(\Om)$ (see~\cite[Theorem 10.18]{Dac}). As a consequence we cannot substitute Lipschitz continuity to the $C^1$-regularity assumption in Proposition~\ref{C1:intro}~(ii). More precisely, 
at least in the range $\min(\te_1,\te_2)\leq1$, when we pass from $C^1$-regularity to Lipschitz continuity, the threshold jumps from $p=\te$ to $p=1+\te$ as shown by the two following propositions.
The first one follows from Theorem~\ref{theointro:zero} (with $\te_1,\te_2$ in the range (b) of~\eqref{Pintro} below).
\begin{propositionINTRO}[Proposition~\ref{propLip}] \label{Lip:intro} 
Assume $\min(\te_1,\te_2)\leq 1$ then for $u\in Lip(\Om)$,  $[\F_{1+\te}(u)=0] \Longrightarrow u\in \U(\Om)$. 
\end{propositionINTRO}
%The proposition is a particular case of Theorem~\ref{theointro:zero} below in the range of parameters (b) of~\eqref{Pintro}. The value $p=1+\te$ is optimal as shown by Proposition~\ref{remintro:optimal} (1)
\begin{propositionINTRO}[Proposition~\ref{rem:optimal}~(i)]\label{remintro:roofoptimal}
There exists  $u\in Lip(\Om)\bks \U(\Om)$ with $0<\F_{1+\te}(u)<\oo$.
\end{propositionINTRO}
The typical example of a function $u\in Lip(\Om)\bks \U(\Om)$ with $\F_{1+\te}(u)<\oo$  is the  ``roof'' function  $u(x):=\min(x_1,x_2)$ defined on $\Om=(0,1)\times(0,1)$ with $n_1=n_2=1$. This function is locally independent of $x_1$ or of $x_2$ away from the diagonal $\{x_1=x_2\}$, so that the integrand in~\eqref{Eeps} vanishes outside $\{(x,z): |x_1-x_2|\le |z|<\eps\}$. With this remark, it is not difficult to guess that $\F_{1+\te}(u)$ 
{ is finite and positive. } 

In this work we consider lower regularities than Lipschitz continuity, like merely measurable functions or $L^1_{loc}$ or $L^\oo$ functions for the finest results (recall however Remark~\ref{remarkLinfty}). In this setting, a second important example of a function which ``almost'' belongs to $\U(\Om)$   is given by the characteristic function of a ``corner'', $u:=\un_{(0,1)^2}$ defined in $\Om=(-1,1)^2$. Here, the integrand of~\eqref{Eeps} vanishes outside $\{(x,z):|x|<\eps,|z|<\eps\}$ and it is easy to check that $\F_{2}(u)<\infty$.
%We refer to Proposition~\ref{rem:optimal} for precise computations.   
\begin{propositionINTRO}[Proposition~\ref{rem:optimal}~(ii)]\label{remintro:corneroptimal}  
There exists  $u\in L^\oo(\Om)\bks \U(\Om)$ with $0<\F_{2}(u)<\oo$.
\end{propositionINTRO}
Propositions~\ref{remintro:roofoptimal} and ~\ref{remintro:corneroptimal} show that the implication 
\[
\forall u\in L(\Om),\quad \F_{p}(u)=0\Longrightarrow u\in \U(\Om)
\] 
could only hold under the condition $p\geq \max(1+\te,2)$. Our first main result shows that in many cases, this bound is sharp.
\begin{definitionINTRO}\label{def:p}
For $\te_1,\te_2>0$, using the notation $\te:=\te_1+\te_2$, we define 
 \be\label{Pintro}
P(\te_1,\te_2)\, :=\, 
\left\{ \begin{array}{clr}
2& \text{ if $\te\leq1$,}&(a)\smallskip\\ 
1+\te& \text{ if $\te\ge 1$ and $\min(\te_1,\te_2)\leq 1$,}&(b)\smallskip\\
\min(\te_1,\te_2)+\te& \text{ if  $\te_1,\te_2>1$.}&(c)
\end{array}\right.
\ee
To lighten notation, from now on, $\te_1,\te_2>0$ being given, we note 
 \[
 \F( u):=\F_{P(\te_1,\te_2)}(u)=\F^{\te_1,\te_2}_{P(\te_1,\te_2)}(u).
 \]
\end{definitionINTRO}
\begin{theoremINTRO}[Theorem~\ref{theo:zero}]\label{theointro:zero}
 For every $u\in L(\Om)$, there holds $\F( u)=0\ \Longrightarrow u\in \U(\Om)$.
\end{theoremINTRO}
\begin{remarkINTRO}\label{remark2}
The counter-examples of Propositions~\ref{remintro:roofoptimal},~\ref{remintro:corneroptimal} show that the exponent $P(\te_1,\te_2)$ cannot  be improved in the cases (a) and (b) of~\eqref{Pintro}. On the contrary, in case $(c)$ we believe that the sharp exponent should still be $1+\te$ although we only succeed to prove that the optimal exponent was not larger than $P(\te_1,\te_2)=\min(\te_1,\te_2)+\te>1+\te$.
 \end{remarkINTRO}
 \medskip
 
Since $\F( u)<\infty$ implies $\F_{p}(u)=0$ for every $p<P(\te_1,\te_2)$ (see Remark~\ref{remark1}), we obtain as direct corollary of Theorem~\ref{theointro:zero}:
\begin{corollaryINTRO}
 If $u\in L(\Om)$ satisfies $\F_{p}(u)<\infty$ for some $p>P(\te_1,\te_2)$, then $u\in \U(\Om)$.
\end{corollaryINTRO}

Another consequence of the theorem is the following generalization of \cite[Proposition~4.3]{GRun}.
\begin{corollaryINTRO}[Corollary~\ref{Coro:geneGRun}]\label{CoroIntro:geneGRun}
 For every $p\ge P(\te_1,\te_2)$, $r>0$, and every { $u\in L(\Om)$, } if 
 \begin{equation*}%\label{hyp:GRun}
   \int_{B_r}\int_{\Om^r} \frac{|u(x+z_1)-u(x)|^{\te_1} |u(x+z_2)-u(x)|^{\te_2}}{|z|^{n+p}} dx dz <\infty,
 \end{equation*}
then $u\in \U(\Om^r)$.
\end{corollaryINTRO}
%\begin{proof}[Proof of Corollary~\ref{CoroIntro:geneGRun}]%Or Corollary~\ref{Coro:geneGRun}\label
%As a first example we see that letting $\rho:= \frac{1}{|B_r\sm B_{r/2}|}\un_{B_r\sm B_{r/2}}$ and for $k\ge 0$, $\eps_k:= 2^{-k} r$, we have 
%\begin{align*}
% \lefteqn{\int_{B_r} \int_{\Om^r} \dfrac{|Du(x,z_1)|^{\te_1} |Du(x,z_2)|^{\te_2}}{|z|^{p+n}} \,dx \,dz}\\
% &=\sum_{k=0}^{\infty} \int_{B_{\eps_k}\sm B_{\eps_{k+1}}} \frac{1}{|z|^n} \int_{\Om^r} \dfrac{|Du(x,z_1)|^{\te_1} |Du(x,z_2)|^{\te_2}}{|z|^{p+n}} \,dx \,dz\\
% &\ge \sum_{k=0}^{\infty}\int_{\R^n} \frac{\un_{B_{\eps_k}\sm B_{\frac{\eps_{k}}{2}}}(z)}{\eps_k^n} \int_{(\Omega^r)^{\eps_k}} \dfrac{|Du(x,z_1)|^{\te_1} |Du(x,z_2)|^{\te_2}}{|z|^{p}} \,dx \,dz\\
%&\ges \sum_{k=0}^{\infty}\int_{\R^n} \rho_{\eps_k}(z)  \int_{(\Omega^r)^{\eps_k}} \dfrac{|Du(x,z_1)|^{\te_1} |Du(x,z_2)|^{\te_2}}{|z|^{p}} \,dx \,dz.
% \end{align*}
%Therefore, if 
%\begin{equation*}%\label{eq:finiteRuna}
% \int_{B_r} \int_{\Om^r} \dfrac{|u(x+z_1)-u(x)|^{\te_1} |u(x+z_2)-u(x)|^{\te_2}}{|z|^{p+n}} \,dx \,dz<\oo,
%\end{equation*}
%then 
%\[
% \lim_{k\up \oo} \int_{\R^n} \rho_{\eps_k}(z)  \int_{(\Omega^r)^{\eps_k}} \dfrac{|Du(x,z_1)|^{\te_1} |Du(x,z_2)|^{\te_2}}{|z|^{p}} \,dx \,dz=0.
%\]
%\end{proof}
%
This is indeed a far reaching generalization of \cite[Proposition~4.3]{GRun} since the same conclusion was obtained there  under the assumptions that $n_2=1$, $u=\un_E$
for some set $E$ of finite perimeter satisfying an extra technical assumption and $p>n$ (see also~\cite{DanRun} where the condition $p>n$ was independently relaxed to $p\ge 2$). As opposed to~\cite{GRun} and~\cite{DanRun}
where the proofs are somewhat geometrical and  based on slicing, our proof is purely analytical. 
 \medskip

The main insight in the proof of Theorem~\ref{theointro:zero} is that if $\F( u)$ does not control  first order differential quotients it does control second order ones.
Very roughly speaking (at least for $\te\le 1$), the main observation is that\footnote{By convention
$a\les b$ means that there exists a non-negative constant $C$ which may only depend on $\te_1,\te_2$, $n$, $\Omega$ or on the kernel $\rho$ such that $a\leq C b$.}  
\begin{equation*}%\label{eq:controlintro}
 \liminf_{\eps\dw 0} \int_{\R^n}\rho_\eps(z) \int_{\Omega^\eps} \dfrac{|u(x+z_1+z_2)-u(x+z_1)-u(x+z_2)+u(x)|^{\te}}{|z|^{P(\te_1,\te_2)}} \,dx \,dz\les \F( u).
\end{equation*}
This yields a  quantitative 
control on  the distribution\footnote{For $u\in L^1_{loc}(\Om)$ we note $\nabla_l$ the distributional gradient with respect to the variables in $X_l$.}
\[
 \mu[u]:=\nabla_1\nabla_2u.
\]
\begin{propositionINTRO}[Proposition~\ref{propquant}]\label{propquant:intro}
Let $u\in L^1_{loc}(\Om)$. For every $\vhi\in C^{\oo}_c(\Om,\R^{n_1\times n_2})$,
 \be
 \label{quanta:intro}
  \langle\mu[u], \vhi\rangle\les 
  \begin{cases}
\quad\quad% \F( u) \|u\|_{\oo}^{1-\te} \|\vhi\|_{\oo} \qquad &\text{ if $\te=1$ or $\te<1$ and  $u\in L^\oo(\Om)$}, \smallskip\\
  \F( u)\|\vhi\|_{\oo} &\text{ if (a) holds with }\begin{cases}\te=1\text{ or }\\ \te<1\text{ and  }\|u\|_\oo
  \leq 1,\end{cases} \medskip\\
\ \F( u)^{1/\te}\|\nb_1 \vhi\|_1^{1-1/\te}\|\vhi\|_{\oo}^{1/\te} &\text{ if (b) holds  and } \te_1\le 1, \medskip\\
\ \F( u)^{1/\te}\|\nb_1 \vhi\|_1^{\te_2/\te}\|\vhi\|_{\oo}^{\te_1/\te} & \text{ if (c) holds  and } \te_1\le \te_2.
  \end{cases}
\ee
\end{propositionINTRO}
\begin{remarkINTRO}
For some estimates (as the first one above), $u$ is required to be bounded. In these situations, we avoid complex formulas involving $\|u\|_\oo$ by assuming $\|u\|_\oo\leq 1$. The general case can be recovered by scaling.
\end{remarkINTRO}
The proof of Theorem~\ref{theointro:zero} continues as follows.  
Proposition~\ref{propquant:intro} implies that if $\F( u)=0$ then $\mu[u]=0$. Obviously, the space of functions $u\in L^1_{loc}(\Om)$ with $\mu[u]=0$ is 
\[
L^1_{loc}(\Om)\cap \Span \U(\Om)=\lt\{x\in \Om\mapsto u_1(x_1)+u_2(x_2) : u_l\in L^1_{loc}(\Om_l)\text{ for }l=1,2\rt\}.
\]
Plugging the decomposition $u(x)=u_1(x_1)+u_2(x_2)$ into the definition of $\E_{\eps,p}(u)$, the integrations over $\Om_1$ and $\Om_2$ decouple. Using again $\E(u)=0$, we deduce that for $l=1$ or $l=2$, we have a  control of the form 
\[
\int_{\Om_l^\eps}\int_{\Om_l^\eps} \tilde \rho_\eps(x_l-y_l) \dfrac{|u_l(x_l)-u_l(y_l)|^{\te_l}}{|x_l-y_l|^{\te_l}} \, dx_l \, dy_l\ \stackrel{\eps\dw0}\longto\ 0.
\]  We then use  ideas from~\cite{BBM2001} in either $\Om_1$ or $\Om_2$ to obtain Theorem~\ref{theointro:zero}.

Before developing further the consequences of Proposition~\ref{propquant:intro} let us comment about its optimality. In the cases (a) and (b) of Proposition~\ref{propquant:intro}, the estimates are optimal in the sense that for $p<P(\te_1,\te_2)$,~\eqref{quanta:intro} does not hold in general. Indeed, we can precise Propositions~\ref{remintro:roofoptimal},~\ref{remintro:corneroptimal} as follows.
\begin{propositionINTRO}[Proposition~\ref{rem:optimal}]\label{remintro:optimal}~
\begin{itemize}
\item[(i)] There exists  $u\in Lip(\Om)\bks \U(\Om)$ with $0<\F_{1+\te}(u)<\oo$ and $\mu[u]\neq 0$ (typically, $u$ is a ``roof'' function);
\item[(ii)] There exists  $u\in L^\oo(\Om)\bks \U(\Om)$ with $0<\F_{2}(u)<\oo$  and $\mu[u]\neq 0$ (typically,  $u$ is a``corner'' function).
\end{itemize}
\end{propositionINTRO}
Let us make the important observation that  Proposition~\ref{propquant:intro} gives more information than $\F(u)=0$ implies $u\in \U(\Om)$.
Indeed, it shows that if $0<\F( u)<\infty$, then $\mu[u]$ lies in some Sobolev space with null or negative regularity exponent. In particular,  in case $(a)$ ($\te=1$ or $\te<1$ and $u\in L^\oo(\Om)$), if $\F( u)=\F_{2}(u)$ is finite then $\mu[u]$ is a finite Radon measure.
{By construction, we prove that it is not necessarily true in the other cases.}
\begin{propositionINTRO}[Proposition~\ref{propCNTEX}]\label{propCNTEX:intro}
For every $\te>1$, there exists $u\in Lip(\R^2)$, compactly supported, with $\F_{1+\te}(u)<\infty$ and for which $\mu[u]=\pt_1\pt_2u$ is not a finite Radon measure.
\end{propositionINTRO}

 %Notice also that if for a function $v$ and a $p>p(\te_1,\te_2)$, $\F_{p}(v)<\infty$ then $\F( v)=0$ and
 %therefore a consequence of~\eqref{quanta} is that for $p>p(\te_1,\te_2)$, if $\F_{p}(v)<\infty$ then $\mu=0$ (and actually that $u=u_1$ or $u=u_2$, see Theorem~\ref{theo:zero} below). 

\subsection{Control of the distance of $u$ to $\U(\Om)$}
 
We then focus on case $(a)$ of~\eqref{Pintro} with the assumption $u\in L^\oo(\Om)$. We prove that the energy gives a quantitative control  on the distance of $u$ to $\U(\Om)$. We obtain the strongest result  in this direction for $n=2$.
\begin{theoremINTRO}[Theorem~\ref{BVestim1d}]\label{BVestim1dintro}
Assume that $n_1=n_2=1$ and $\te\le 1$. Then, for every $u\in L^\oo(\Om)$ with $\|u\|_\oo\le 1$ and $\F( u)<\oo$, there exists $\bar u\in \U(\Om)$  such that $u-\bar u\in BV(\Om)\cap L^\oo(\Om)$ with the estimate
 %\begin{multline}
 \begin{equation*}
 %\label{quantrig2dintro}
 \|u-\bar u\|_\oo+|\nb[u-\bar u]|(\Om)
 \les    \F( u)+\F( u)^{\frac{1}{2}}.
 %\les \lt( \|u\|_\oo^{1-\te}\F( u)+\|u\|_\oo^{\te_1+(1-\te)(1+\frac{1}{2}(\te_2-\te_1))}\F( u)^{\te_2+\frac{1}{2}(1-\te)}\lt.\\ +\|u\|_\oo^{\te_2+(1-\te)(1+\frac{1}{2}(\te_1-\te_2))}\F( u)^{\te_1+\frac{1}{2}(1-\te)}\rt). \end{multline}
 \end{equation*}
\end{theoremINTRO}
%In the case $\te<1$,~\eqref{quantrig2dintro} may be improved: see~\eqref{introimproved} below. %\textcolor{red}{est ce qu on explique un peu la preuve?}
The idea of the proof is to first decompose $u$ as $u(x)=u_1(x_1)+u_2(x_2)+w(x)$ where $w$ satisfies $\partial_1\partial_2 w=\mu[u]=\partial_1\partial_2 u$ and $\|w\|_\infty+|\nabla w|(\Omega)\les \F(u)$.
Using this (and in particular the $L^\infty$ bound on $w$), we can quantify how much the integration with respect to $x_1$ and $x_2$ in the definition  \eqref{Eeps} of $\E_{\eps,2}(u)$ decouples. 
In higher dimension, the failure of the Sobolev embedding $BV(\Om_l) \subset L^\infty(\Om_l)$ makes the situation more complex (and in particular { the energy does not control the corresponding $w$ in $L^\infty$}) and we were not able to obtain a $BV$ estimate. Nevertheless, we have, 
\begin{theoremINTRO}[Theorem~\ref{coro:u-baru}]\label{thm:LpestimIntro}
Assume that $\hat n:=\max(n_1,n_2)\geq 2$, that $\Om_1$ and $\Om_2$ are bounded extension domains and that $\te\le1$. Let $u\in L^\oo(\Om)$ with $\|u\|_\oo\leq 1$ and $\F( u)<\oo$.
Then, there exists $\bar u\in L^\oo(\Om)\cap \U(\Om)$ such that
 \begin{equation*}%\label{estimbaruintro}
  \|u-\bar u\|_{L^{\frac{\hat{n}}{\hat{n}-1}}(\Om)} \les\, \F( u)+\F( u)^{\frac{1}{2}}. % \lt((\|u\|_{\oo}^{1-\te}\F( u)^{1/2} + \|u\|_{\oo}^{1-\te}\F( u)\rt) \lt(1+\|u\|_\oo\rt).
 \end{equation*}
\end{theoremINTRO}
Notice that the $L^{\frac{\hat{n}}{\hat{n}-1}}$ norm (which comes from the embedding of  $BV(\Om_l)$ into $L^\frac{n_l}{n_l-1}(\Om_l)$) is stronger than the $L^{\frac{n}{n-1}}$ norm which would come from the embedding of $BV(\Om)$.

\subsection{Further results}
In a second paper~\cite{GM2}, we {will} focus on the case $\te\leq 1$ with $u\in L^\oo(\Om)$ and $\F(u)<\oo$ and study the structure of the defect measure $\mu=\mu[u]$ (which is  then a Radon measure).

 In dimension 2 (\textit{i.e.} $n_1=n_2=1$) we show that if $\te=1$, then $\mu$ concentrates on a set with Hausdorff dimension {at most} 1. Moreover, if  $u$ is Lipschitz continuous, then $\mu$ is absolutely continuous with respect to the Hausdorff measure $\h^1$ (and satisfies the differential inclusion~\eqref{diffinclu}). 
 On the contrary, if $\te<1$, we show that $\mu$ concentrates on a countable set: there exist sequences $(x_k)_{k\geq1}\subset\Om$ and $(m_k)_{k\geq 1}\subset\R$ such that $\sum |m_k|^\te\les \F_2(u)$ and $\mu=\sum m_k\delta_{x_k}$.
Moreover, if $u$ is a characteristic function, then $m_k\in\pm \{1,2\}$ for every $k\geq1$. As a consequence of the estimate  $\sum |m_k|^\te\les \F_2(u)$ 
 there exists some $\eta>0$ such that $\F_2(u)<\eta$ implies $\mu=0$, which in turn leads to $u\in S(\Om)$. This improves Theorem~\ref{theointro:zero} and Theorem~\ref{BVestim1dintro} in this particular case.

In higher dimensions, we assume $\te<1$.  Using tools from Geometric Measure Theory (mainly the rectifiability criterion for flat chains of White~\cite{White1999}), we prove that $\mu$ is a $(n-2)$-rectifiable measure with a tensor structure: for $l=1,2$, there exist $\Sigma_l\subset X_l$
 $(n_l-1)-$rectifiable and  a Borel function $m:\Sigma_1+\Sigma_2\to\R$ such that $\mu=m \nu_1\otimes \nu_2\h^{n-2}\restr (\Sigma_1+\Sigma_2)$, where { for $l=1,2$, $\nu_l\in X_l$   is a normal} to $\Sigma_l$.
This gives a relatively good understanding of the case where the typical function $u$ with $\F(u)<\infty$ is a ``corner'' (recall Proposition~\ref{remintro:optimal}).\medskip

In the case  $\te=1$, in order to distinguish between ``corners'' and ``roofs'' (see Proposition~\ref{remintro:optimal}), one needs to impose more regularity on $u$. To understand this better,
 we plan to investigate in a future work,   the set of Lipschitz continuous functions  which satisfy 
  the differential inclusion~\eqref{diffinclu} and are such that $\mu[u]=\nabla_1\nabla_2u $ is a Radon measure. 
%In summary:
% \be\label{setGM3}
% \nb u(x)\in [-1,1]\times\{0\} \cup \{0\}\times[-1,1]\text{ almost everywhere in }\Om\ \text{ and }\ \pt_1\pt_2u\in \mathcal{M}(\Om).
% \ee
% These assumptions fit the result of~\cite{GM2} where $u$ is Lipschitz continuous with $\F_2(u)<\oo$, $\te=1$. We precise the result of~\cite{GM2} by showing that $\mu[u]$ is a 1-rectifiable measure, {\red \dots, traces of $u$, l..s.c. energies, limit model by homogenization \dots ?}

\subsection{Conventions and notation}
In all the paper, we consider $\te_1,\te_2>0$ and note $\te=\te_1+\te_2$ their sum. For $x,z\in \R^n$, and $u:\Om\to\R$ we note 
\[
Du(x,z):=u(x+z)-u(x).
\]  
In particular, the expression~\eqref{Eeps} of $\E_{\eps,p}^{\te_1,\te_2}(u)$ simplifies to
\[
\E_{\eps,p}^{\te_1,\te_2}(u)=\int_{\R^n}\rho_\eps(z) \int_{\Omega^\eps} \dfrac{|D u(x,z_1)|^{\te_1} \, |Du(x,z_2)|^{\te_2}}{|z|^p} \,dx \,dz.
\]

As already said, except at two points (case (c) in the proofs of Proposition~\ref{propquant:intro} and Theorem~\ref{theointro:zero}),
we omit the superscript $\te_1,\te_2$ by writing $\E_{\eps,p}$ for $\E_{\eps,p}^{\te_1,\te_2}$ and $\F_{p}$ for $\F_{p}^{\te_1,\te_2}$. { When $p=P(\te_1,\te_2)$ as defined in~\eqref{Pintro},} we simply write $\F$ for $\F_p$. 

If $v\in L(\Om)$ and $U=U_1+U_2$ with $U_l\subset \Om_l$ a subdomain of $\Om_l$ for $l=1,2$,  we note,
\[
\E_{\eps,p}(v;U),\text{ for }\eps>0,\quad\text{and}\quad \F_{p}(v;U)
\]
the energies of the restriction $v_{|U}$.

%\[
%V=\Span G=\{ u_1(x_1)+u_2(x_2):\ u_1:\Om_1\to\R, \ u_2:\Om_2\to\R \text{ measurable functions}\}, 
%\]

In the sequel $(e_1,\cdots, e_{n_1})$ (respectively $(f_1,\cdots, f_{n_2})$) denotes an orthonormal basis of $X_1$ (respectively of $X_2$).

Given $x\in\R^n$, we note $x_l$  its orthogonal projection on $X_l$, for $l=1,2$, so that $x=x_1+x_2$.  Similarly, given a function $\varphi \in C^\oo_c(\Om)$ and $x\in \Om$, we note $\nb_1 \vhi(x) +\nb_2\vhi(x)$ the decomposition of $\nb \vhi(x)$ in $X_1+X_2$. By duality, the distributional derivative of a function  $u\in L^1_{loc}(\Om)$ decomposes as $\nb u=\nb_1 u+\nb_2 u$, where $\nb_l u $ is a distribution on $\Om$ with values into $X_l$.

{ For $r>0$ and $x\in \R^n$, $B_r(x)$ denotes the open ball in $\R^n$ with center $x$ and radius $r$. If $x=0$, we simply write $B_r$. For $l=1,2$, we note $B^{X_l}_r(x_l)$ the open ball in $X_l$ with center $x_l$ and radius $r$. }

We use standard notation for the function spaces ($L^1_{loc}(\Om)$, $L^q(\Om)$, $L^\oo(\Om)$, $Lip(\Om)$, $BV(\Om)$).

For $a\in \R$, we note $a_+:=\max(a,0)$ its non-negative part.

%We note $\h^1$ the 1-dimensional Hausdorff measure and 
If $A\subset \R^n$ is a measurable subset, we note $|A|$ its volume and $\h^k(A)$ its $k-$dimensional Hausdorff measure.

We note $\smint_\Om u(x) \, dx =|\Om|^{-1}\int_\Om u$ the mean value of a function $u$ over $\Om$.

Eventually, by convention, $a\les b$ means that there exists a non-negative constant $C$ which may only depend on $\te_1,\te_2$, $n$, $\Omega$ or on the kernel $\rho$ such that $a\leq C b$.

% We let $\mathcal{M}_{n_1,n_2}$ be the set of matrices of size $n_1\times n_2$.

\subsection{Outline of the paper}
In the first section, we prove Proposition~\ref{C1:intro} and we build all the examples and counter-examples introduced above. In Section~\ref{S2}, we consider the zero-energy case and prove Theorem~\ref{theointro:zero} (and all the results from Proposition~\ref{Lip:intro} to~\ref{propCNTEX:intro}).  Section~\ref{S3} is dedicated to the quantitative control of the distance of $u$ to $\U(\Om)$ in terms of $\F( u)$.% In Section~\ref{S4} we restrict the study to the cases $\te\le 1$ and $n_1=n_2=1$ and we study the structure of the defect measure $\mu[u]$ when $\F( u)<\oo$. The higher dimensional case is treated in Section~\ref{S5}. {\red Appendix ??}

\section{The case of $C^1$-functions and counter-examples}\label{Sexamples}
We first settle Question~\eqref{question} in the setting of $C^1$-functions.
\begin{proposition}[Proposition~\ref{C1:intro}]\label{prop:C1} ~
\begin{itemize}
\item[(i)] For every $u\in \Lip(\Om)$, $\F_{\te}(u)<\oo$ (and therefore $\F_{p}(u)=0$ for $p<\te$);
\item[(ii)] For $u\in C^1(\ov {\Om})$, we have  $\F_{\te}(u)=0 \Longrightarrow u\in \U(\Om)$. 
\end{itemize}
\end{proposition}
\begin{proof}
(i) Let $u\in Lip(\Om)$ with Lipschitz constant $\lambda$. We have 
\[
\E_{\eps,\te}(u)\leq \lambda^\te \int_{\Omega^\eps} \lt[\int_{\R^n}\dfrac{|z_1|^{\te_1} \, |z_2|^{\te_2}}{|z|^\te} \rho_\eps(z) \, dz \rt] \,dx.
\]
Using the change of variable $z=\eps \tilde z$ and sending $\eps$ to 0, we get
\[
\F_{\te}(u)\leq \lambda^\te |\Om| \lt(\int_{\R^n} \dfrac{|z_1|^{\te_1} \, |z_2|^{\te_2}}{|z|^\te} \rho(z) \, dz\rt) ,
\]
which proves (i).\medskip
 
(ii) Let $u\in C^1(\ov \Om)$. Plugging the relation $u(x+z_l)-u(x)=\nabla_l u(x)\cdot z_l+o(|z_l|)$ into the definition of $\E_{\eps,\te}^{\te_1,\te_2}(u)$, we observe that
\[
\E_{\eps,\te}^{\te_1,\te_2}(u)= \int_{\Omega^\eps} \lt[\int_{\R^n}\lt( \dfrac{|\nb_1 u(x)\cdot z_1|^{\te_1} \, |\nb_2 u(x)\cdot z_2|^{\te_2}}{|z|^\te} +o(|z|) \rt)\rho_\eps(z) \, dz \rt] \,dx.
\]
Sending $\eps$ to 0, we have
\be\label{introC1_1}
\F_{\te}^{\te_1,\te_2}(u)=\lt(\int_{\R^n} \dfrac{|v_1\cdot z_1|^{\te_1}\, |v_2\cdot z_2|^{\te_2}}{|z|^{\te}}\rho(z)\, dz\rt)\int_\Om |\nb_1 u(x)|^{\te_1} |\nb_2 u(x)|^{\te_2}\, dx,
\ee
where $v_j\in X_l$ are (arbitrary)  unit vectors.\\
Hence, if $\F_{\te}^{\te_1,\te_2}(u)=0$, \eqref{introC1_1} implies that $|\nb_1 u(x)||\nb_2u(x)|=0$ in $\Om$. \medskip

Next, suppose that at some point $x^0\in\Om$, $w_1:=\nb_1 u(x^0)\neq0$. Let $C_2$ be the connected component of the set 
\[
\{y_2\in \Om_2 : \nb_1 u(x_1^0+y_2) =w_1\}
\]
which contains $x_2^0$. By continuity of $\nb_1 u$, $C_2$ is closed in $\Om_2$. Let us show that it is also open. Let $y_2\in C_2$,  by continuity of $\nb_1 u$ we have $\nb_1 u\not\equiv 0$ in some neighborhood $N_\eps(x^0_1+y_2):=B^{X_1}_\eps(x^0_1)+B^{X_2}_\eps(y_2)$. Since $|\nb_1 u||\nb_2u|$ vanishes, we have $\nb_2u\equiv 0$ in $N_\eps (x^0_1+y_2)$. In particular, $\nb_1 u(x_1^0+y'_2)=\nb_1 u(x_1^0+y_2)=w_1$ for every $y'_2\in B^{X_2}_\eps(y_2)$. We conclude that $B^{X_2}_\eps(y_2)\subset C_2$ which proves that $C_2$ is relatively open in $\Om_2$. Finally, the open set  $\Om_2$ being connected, we have $C_2=\Om_2$ and we conclude that $\nb_1 u(x^0_1+y_2) \neq 0$ for every $y_2\in \Om_2$. 

Assuming by contradiction that there exists also some point $\tilde x^0\in \Om$ with $\nb_2 u(\tilde x^0)\neq0$ we obtain similarly that $\nb_2 u(y_1+\tilde x^0_2) \neq 0$ for every $y_1\in \Om_1$. Choosing $y_1=x^0_1$ and $y_2=\tilde x^0_2$ we obtain at the point $x=x^0_1+\tilde x^0_2$, the contradiction 
\[
\nb_1 u(x)\neq0\text{ and }\nb_2 u(x)\neq0.
\]
We conclude that   $\nb_lu\equiv 0$ for $l=1$ or $l=2$: in short, $u\in \U(\Om)$.  This proves (ii).
\end{proof}
We now, give lower-bounds for $p$ in Question~\eqref{question} in the setting of Lipschitz continuous and of bounded functions.
\begin{proposition}[Proposition~\ref{remintro:optimal}]\label{rem:optimal}~
\begin{itemize}
\item[(i)] There exists  $u\in Lip(\Om)\bks \U(\Om)$ with $0<\F_{1+\te}(u)<\oo$ and $\mu[u]\neq 0$;% (typically, $u$ is a ``roof'' function).
\item[(ii)] There exists  $u\in L^\oo(\Om)\bks \U(\Om)$ with $0<\F_{2}(u)<\oo$  and $\mu[u]\neq 0$.% (typically,  $u$ is ``corner'' function).
\end{itemize}
\end{proposition}
\begin{proof}
(i) 
Let us  assume that  $n_1=n_2=1$ and let us introduce the ``roof'' function defined on $\R^2$  by $v(x_1,x_2):=\min(x_1,x_2)$. We have $\partial_1\partial_2 v=(1/\sqrt2) \h^1\restr \{x_1=x_2\}$. We now consider $\Om=(0,1)^2$ and $u(x):=v(x)$ in $\Om$. We have $\mu[u]=(1/\sqrt2) \h^1\restr J$, with $J=\{(t,t):t\in (0,1)\}$ and $\mu[u]$ is a Radon measure with
\[
0< |\mu[u]|(\Om)<\oo.
\]

Next, using the changes of variable $x=\eps x'$, $z=\eps z'$ and the 1-homogeneity of $u$, we obtain for $\eps>0$,
\[
\E_{\eps,1+\te}(u)=\eps \,\E_{1,1+\te}\lt(v;\frac1\eps \Om\rt).
\]
Let us set $w(s):=\min(s,0)$. We have 
\[
v(x_1+z_1,x_2)-v(x)=D w(x_1-x_2,z_1),\qquad v(x_1,x_2+z_2)-v(x)=D w(x_2-x_1,z_2)
\]
and thus
\be\label{Eepsv}
\E_{1,1+\te}\lt(v;\frac1\eps \Om\rt) =\int_{\lt[\frac1\eps \Om\rt]^1}  \lt[\int_{\R^2}\rho(z) \dfrac{|D w(x_1-x_2,z_1)|^{\te_1}\, |D w(x_2-x_1,z_2)|^{\te_2}} {|z|^{1+\te}} \,dz\rt]\, dx.
\ee
We claim that $\E_{1,1+\te}\lt(v;\frac1\eps \Om\rt) =c_1/\eps-c_2$ for some $c_1,c_2>0$. The idea is that since $w$ is constant in $\R_+$,
the integrand vanishes when $|z|\leq |x_1-x_2|$ and since $\supp \rho\subset B_1$, the integrand vanishes away from the $\sqrt{2}/2$-neighborhood of the diagonal segment $(1/\eps)J$.
Since this segment has length $\sqrt{2}/\eps$, we expect $\E_{1,1+\te}\lt(v;\frac1\eps \Om\rt) $ to be of order $1/\eps$ and therefore, $\E_{\eps,1+\te}(u)$ to be of order 1.

Let us note $f(x_1-x_2)$ the integral inside the brackets in~\eqref{Eepsv}. The function $f$ takes values in $[0,\infty]$,  is measurable and even. We first 
show that $f$ is integrable on $(-1,1)$. For $s\geq 0$ and $z\in \R^2$, we have 
\be\label{|Dw|}
|Dw(s,z_1)|^{\te_1}\,|Dw(-s,z_2)|^{\te_2}=
\begin{cases} 
|s+z_1|^{\te_1} |z_2|^{\te_2} &\text{ if }z_2\leq s\leq -z_1,\smallskip \\
|s+z_1|^{\te_1} s^{\te_2} &\text{ if }0<s<\min(z_2,-z_1),\smallskip \\
\qquad 0&\text{ in the other cases.}
\end{cases}
\ee
Integrating in $s$ over $(0,1)$, we deduce 
\[
\int_0^1 |Dw(s,z_1)|^{\te_1}\,|Dw(-s,z_2)|^{\te_2}\, ds\les |z|^{1+\te}.
\]
%In the range $0<s<z':=\min(z_2,-z_1)$, we have, 
%\[\int_0^{z'} |Dw(-s,-z_1)|^{\te_1}\,|Dw(s,-z_2)|^{\te_2}\, ds\leq  (|z_1|+z_2)^{\te_1} \int_0^{z_2} s^{\te_2}\, ds\, =\,  (|z_1|+z_2)^{\te_1} \dfrac{z_2^{\te_2}}{1+\te_2} \les |z|^{1+\te}.
%\]
%In the range $z_2\leq s\leq -z_1$, noting $I=(\max(0,z_2),-z_1)$, we have similarly,
%\[
%\int_I |Dw(-s,-z_1)|^{\te_1}\,|Dw(s,-z_2)|^{\te_2}\, ds\leq  |z_2|^{\te_2} \int_{z_2}^{-z_1} (-s-z_1)^{\te_2}\, ds\,  \les |z|^{1+\te}.
%\]
Using Fubini and the symmetries of the problem and the above estimate, we get,
\[
\int_{-1}^1 f(s) ds\, =2\int_{\R^2}\dfrac{\rho(z)}{|z|^{1+\te}} \lt[\int_0^1|Dw(s,z_1)|^{\te_1}\,|Dw(-s,z_2)|^{\te_2} \,ds\rt] \,dz \les \int_{\R^2}\rho(z)\, dz =  1.
\]
Therefore, $f\in L^1(-1,1)$. We also observe from~\eqref{|Dw|} that the integral is strictly positive. In summary,
\be\label{c1=intf}
c_1:=\int_{-1}^1 f(s) ds\, \in (0,\infty).
\ee
 Let us return to~\eqref{Eepsv}. Since $\supp\rho\subset B_1$, the identity~\eqref{|Dw|} shows that we can reduce the integration with respect to  $x$ to the set 
\[
A_\eps:=\lt\{(x_1,x_2)\in \R^2 :1<x_1,x_2<1/\eps-1,\ |x_1-x_2|<1\rt\}.
\]
Let us introduce the following rectangle which contains $A_\eps$, 
\[
R_\eps:=\lt\{x\in\R^2;0<x_1+x_2<2/\eps,\, |x_1-x_2|<1\rt\}
\]
We also set $C_\eps:=R_\eps\bks\ov A_\eps$, see Figure~\ref{figReps}.
\begin{figure}[H]
\centering
\begin{tikzpicture}[scale=.8]
\draw (0,0) rectangle (10,10);
\draw (1.6,9.4)node{$\frac1\eps\Om$};
\draw (-.3,-.3)node{$0$};
\draw (0,.98)node{---};
\draw (-.5,1)node{$1$};
\draw (1,0)node{$|$}; 
\draw (1,-.5)node{$1$};
\draw (0,8.98)node{---};
\draw (-.8,9)node{$\frac1\eps-1$};
\draw (0,9.98)node{---};
\draw (-.5,10)node{$\frac1\eps$};
\draw (9,0)node{$|$}; 
\draw (9,-.5)node{$\frac1\eps-1$};
\draw (10,0)node{$|$}; 
\draw (10,-.5)node{$\frac1\eps$};
\draw[very thin,dashed,color=gray]  (1,0) -- (1,10);
\draw[very thin,dashed,color=gray]  (9,0) -- (9,10);
\draw[very thin,dashed,color=gray]  (0,1) -- (10,1);
\draw[very thin,dashed,color=gray]  (0,9) -- (10,9);
\fill[pattern=north east lines, pattern color=white!90!black] (1,1) rectangle (9,9);
\draw (2.9,7.1)node{$\lt[\frac1\eps\Om\rt]^1$};
\draw[color=RedOrange, pattern=north west lines, pattern color=RedOrange] (-.5,.5)--(.5,-.5)--(2,1)--(1,1)--(1,2)--cycle;
\draw[color=RedOrange, pattern=north west lines, pattern color=RedOrange] (10.5,9.5)--(9.5,10.5)--(8,9)--(9,9)--(9,8)--cycle;
\fill[pattern=north west lines, draw=Cerulean, pattern color=Cerulean] (1,1)--(2,1)--(9,8)--(9,9)--(8,9)--(1,2)--cycle;
\draw[color=Cerulean] (5.9,4.1)node{$A_\eps$};
\draw[color=RedOrange] (2,.5)node{$C_\eps$};
\draw[color=RedOrange] (9.5,8)node{$C_\eps$};
\end{tikzpicture}
\caption{\label{figReps}The sets $A_\eps$ and $C_\eps=R_\eps\bks \ov{A}_\eps$.}
\end{figure}
We decompose,
\[
\E_{1+\te}\lt(v;\frac1\eps \Om\rt) =\int_{R_\eps} f(x_1-x_2)\, dx -\int_{C_\eps} f(x_1-x_2)\, dx. 
\]
Obviously the second integral does not depend on $\eps$ and is positive, we note $c_2$ its value. For the first integral, 
we perform the change of variable $s=x_1-x_2$, $t=(1/2)(x_1+x_2)$ to obtain 
\[
\int_{R_\eps} f(x_1-x_2)\, dx =\eps^{-1} \int_{-1}^1 f(s)\, ds \, =  c_1/ \eps.
\]  
This gives $\E_{1+\te}\lt(v;\frac1\eps \Om\rt) =c_1/\eps -c_2$, which leads to
\[
\E_{\eps,1+\te}(u)=c_1-\eps c_2.
\]
Sending $\eps$ to 0, we get $\F_{1+\te}( u)=c_1\in (0,\infty)$ (recall \eqref{c1=intf}) and conclude the proof of (i)  in the case $n_1=n_2=1$ and $\Om=(0,1)^2$. We obtain counter-examples for any non empty and bounded two dimensional domain by translation and scaling of this example and in higher dimensions by extending 
the constructions trivially in the complementary directions.\medskip

(ii) Again, we only treat the two-dimensional case $n_1=n_2=1$, since higher dimensional cases may be obtained by tensorisation. 
Let $\Om=(-1,1)^2$ and set $u:=\un_\om$ with $\om:=(0,1)^2$. Obviously, $\mu=\partial_1\partial_2 u=\delta_{0}$ is a non-trivial finite Radon measure. Let us compute $\F_{2}(u)$.
First, we have
\[
|D u (x,z_1)|^{\te_1}|Du(x,z_2)|^{\te_2}=\begin{cases} 1  & \text{ if }  x_1>0,\ x_2>0,\ x_1+z_1\leq 0\text{ and }  x_2+z_2\leq 0.
 \\
 0 & \text{ in the other cases.}
\end{cases}
\]
We compute for $0<\eps<1$,
 \[
  \E_{\eps,2}(u)= \int_{\R_-\times \R_-} \dfrac{\rho_\eps(z)}{|z|^2}
  \lt[\int_0^{-z_1}\int_0^{-z_2} \, dx_1\, dx_2\rt]\, dz =\dfrac14 \int_{\R^2} \dfrac{\rho_\eps(z) |z_1|\,|z_2|}{|z|^2}\, dz= \dfrac1{4\pi}.
  \]
For the last identity, we have used the radial symmetry of  $\rho_\eps$ and $\int_{\R^2} \rho_\eps=1$. We conclude that $\F_2( u)= 1/(4\pi)\in(0,\oo)$  as required.
\end{proof}
%Notice that in this proof we did not use the fact that $\min(\te_1,\te_2)\le1$ so that it is tempting to conjecture that the optimal exponent in case (c) of~\eqref{P} is again $1+\te$ instead of $\min(\te_1,\te_2)+\te$.\medskip

In Proposition~\ref{propquant} below, we will prove that  in the case $\te\le 1$, if $\E(u)$ is finite then $\mu[u]$ is a Radon measure. We show that this is  no longer true when $\te>1$.
 \begin{proposition}[Proposition~\ref{propCNTEX:intro}]\label{propCNTEX}
{ For every $\te_1,\te_2>0$ with $\te=\te_1+\te_2>1$, there exists $u\in L^\oo(\R^2)$ compactly supported, with $\F_{1+\te}(u)<\infty$ but for which $\mu[u]=\pt_1\pt_2u$ is not a finite Radon measure.}
\end{proposition} 
\begin{proof}
 We will build on the example of Proposition~\ref{rem:optimal}~(i) and consider the ``hat'' function defined on $(0,1)^2$ as  $w(x_1,x_2):=\min(x_1,x_2,1-x_1,1-x_2)$ and extended by zero outside $(0,1)^2$.  Arguing as in the proof of Proposition~\ref{rem:optimal}, we have $0<\F_{1+\te}( u)<\infty$.
 For $h, \ell>0$, let $w_{h,\ell}(x):=hw(\ell^{-1}x)$ be defined on $\R^2$. We then have  
 \be\label{eq:rescale}
 |\partial_1\partial_2 w_{h,\ell}|((0,\ell)^2)=h |\partial_1\partial_2 w|((0,1)^2)=2h \quad \text{ and }\quad \F_{1+\te}( w_{h,\ell})=h^{\te} \ell^{1-\te} \F_{1+\te}( w).
 \ee
 Let $(h_k)_{k\ge 1}$, $(\ell_k)_{k\ge 1}\subset(0,\infty)$ be two decreasing sequences such that
 \be\label{condetah}
\lt(\dfrac{h_k}{\ell_k}\rt)\ \text{is decreasing,}
\quad \dfrac{h_k^\te}{\ell_k^{1+\te}}\sum_{j\geq k} \ell_j^2\,\stackrel{k\up\oo}\longto\,0,
%\quad  \ell_{k}\sum_{j= 0}^{k} \lt(\dfrac{h_j}{\ell_j}\rt)^\te\,\stackrel{k\up\oo}\longto\,0,
\quad
\sum_{k}  h_k^{\te} \ell_k^{1-\te} <\oo\quad\text{ but }\quad \sum_k h_k=\oo.
 \ee
 We also require that there exists some constant $c>0$ such that 
 \be
 \label{condetah2}
 \ell_{k+1}\ge c\,\ell_k \quad\text{for }k \ge 1,
 \ee
 and that 
 \be
 \label{condetah3}
\ell_0 \leq 1/2.
 \ee
 For instance, the sequences defined by  $h_k:=k^{-(4\te+1)/(5\te)}$ and  $\ell_k:=k^{-3/4}/2$  satisfy~\eqref{condetah},~\eqref{condetah2} and~\eqref{condetah3}.
%\footnote{notice that if $\sum_k h_k^{\te} \ell_k^{1-\te}<\infty$  but $\sum_k h_k=\oo$, Young's inequality implies that $\sum_k \ell_k=\infty$.} $\sum_k h_k=\oo$),
Next, we build a bounded sequence $(x^k)_{k\geq 0}\subset \R^2$ such that 
 \be\label{condxk}
 \inf_{j\neq k} |x^j-x^k|\ge 2\ell_k\quad\text{ for every }k\geq0.
 \ee
For this, we observe that $\sum \ell_k\geq (\ell_1/h_1) \sum h_k=\infty$. We set $k_0=1$ and we define recursively  $k_{m+1}$ for $m=0,1,2,\cdots,$ as the unique integer to be such that 
 \[
1-2\ell_{k_{m+1}} < \sum_{k=k_m}^{k_{m+1}-1} 2 \ell_k \le 1.
 \]
 Notice that by~\eqref{condetah3}, the sequence $(k_m)_{m\geq0}$ is not stationary (and thus converges towards $+\oo$).
 We then define  $L_{0}:=0$ and for $q\geq 1$, $L_q:=\sum_{m=0}^{q-1}2\ell_{k_m}$. The sequence $(L_q)_{q\geq 0}$ is increasing and bounded, indeed
 \begin{align*}
 L_\oo:=\lim_{q\up\oo} L_q&=2\ell_1 + \sum_{m\geq 1}^{\oo}2\ell_{k_m}
 \le 2\ell_1 + \sum_{m\geq 0}^{\oo}2\ell_{k_{m+1}} \lt(\sum_{k=k_m}^{k_{m+1}-1} 2 \ell_k\rt)\\
 &\leq 2\ell_1 +  4 \sum_{m\geq 0}^{\oo} \lt(\sum_{k=k_m}^{k_{m+1}-1}  \ell_k^2\rt) =2\ell_1 +  4 \sum_{k\geq 0}  \ell_k^2 \ \stackrel{\eqref{condetah}}<\ \oo.
 \end{align*}
 Now, we set for $m\geq 0$ and $k_m\le k <k_{m+1}$, 
 \[
 x^k:=\lt(L_m, \sum_{k_m\le j < k}2\ell_j\rt).
 \]
 By construction the sequence $(x^k)$ satisfies~\eqref{condxk} and $(x^k)\subset (0,L_\oo)\times (0,1)$. Eventually, we define our candidate (see Figure~\ref{figuArray}), 
 \[
  u(x):=\sum_{k\ge 1} w_{h_k,\ell_k}(x-x^k)\quad\text{for }x\in \R^2.
 \]
 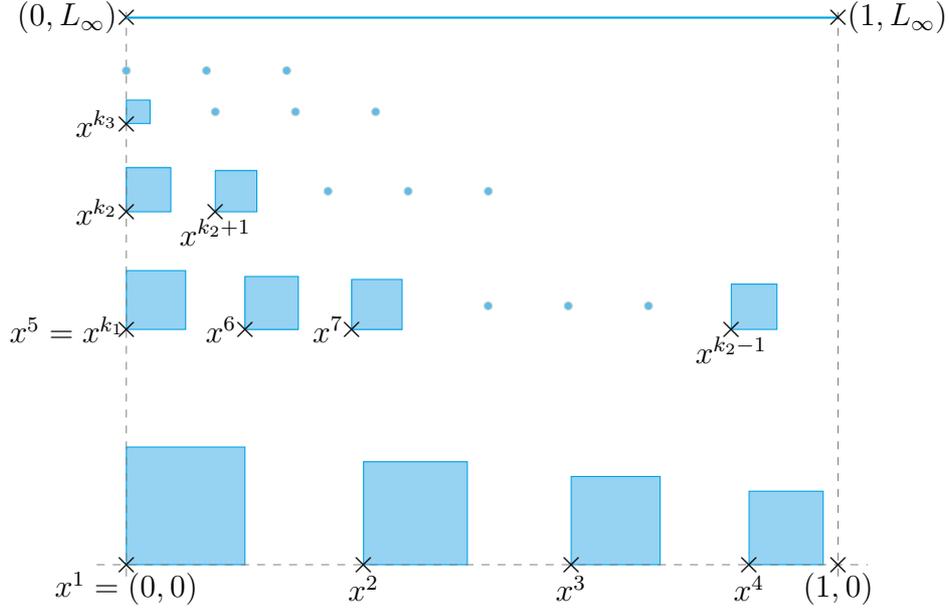
\begin{figure}[H]
\centering
\begin{tikzpicture}[scale=.78, 
decoration={
	markings, mark=between positions 0 and 1 step 30pt with 
		{
		\fill[color=white!40!Cerulean, draw=white!60!gray]  circle[radius=1.5pt];
		}
			}]	
\draw[very thin,dashed,color=gray] (-.5,0)-- (12.5,0); % horizontal dashed line
\draw (12,0)node{$\times$}; 
\draw (12,-.4)node{$(1,0)$};
\draw[very thin,dashed,color=gray] (0,-.2)-- (0,9.5);     %dashed lines
\draw[very thin,dashed,color=gray] (12,-.2)-- (12,9.5); %vertical
\draw (0,9.3)node{$\times$};
\draw (-1,9.3)node{$(0,L_\oo)$};
\draw[thick, color=Cerulean] (0,9.3)--(12,9.3);% Upper line
\draw (12,9.3)node{$\times$};
\draw (13,9.3)node{$(1,L_\oo)$};
%% First line of squares
\draw (0,0)node{$\times$};
\draw (0,-.4)node{$x^1=(0,0)$};
\draw (4,0)node{$\times$};
\draw (4,-.4)node{$x^2$};
\draw (7.5,0)node{$\times$};
\draw (7.5,-.4)node{$x^3$};
\draw (10.5,0)node{$\times$};
\draw (10.5,-.4)node{$x^4$};
\begin{pgfonlayer}{bg} 
\fill[color=white!60!Cerulean, draw=Cerulean] (0,0) rectangle (2,2);
\fill[color=white!60!Cerulean, draw=Cerulean] (4,0) rectangle (5.75,1.75);
\fill[color=white!60!Cerulean, draw=Cerulean] (7.5,0) rectangle (9,1.5);
\fill[color=white!60!Cerulean, draw=Cerulean] (10.5,0) rectangle (11.75,1.25);
\end{pgfonlayer}
%% Second line of squares 
\draw (0,4)node{$\times$};
\draw (-1,4)node{$x^5=x^{k_1}$};
\draw (2,4)node{$\times$};
\draw (1.6,4)node{$x^6$};
\draw (3.8,4)node{$\times$};
\draw (3.4,4)node{$x^7$};
\draw (10.2,4)node{$\times$};
\draw (10.2,3.65)node{$x^{k_2-1}$};
\begin{pgfonlayer}{bg} 
\fill[color=white!60!Cerulean, draw=Cerulean] (0,4) rectangle (1,5);
\fill[color=white!60!Cerulean, draw=Cerulean] (2,4) rectangle (2.9,4.9);
\fill[color=white!60!Cerulean, draw=Cerulean] (3.8,4) rectangle (4.65,4.85);
\fill[color=white!60!Cerulean, draw=Cerulean] (10.2,4) rectangle (10.97,4.77);
\end{pgfonlayer}
\path[postaction=decorate] (6.1,4.4) -- (9.1,4.4);
%% Third line of squares 
\draw (0,6)node{$\times$};
\draw (-.5,6)node{$x^{k_2}$};
\draw (1.5,6)node{$\times$};
\draw (1.5,5.65)node{$x^{k_2+1}$};
\begin{pgfonlayer}{bg} 
\fill[color=white!60!Cerulean, draw=Cerulean] (0,6) rectangle (.75,6.75);
\fill[color=white!60!Cerulean, draw=Cerulean] (1.5,6) rectangle (2.2,6.7);
\end{pgfonlayer}
\path[postaction=decorate] (3.4,6.35) -- (6.4,6.35);
%% Fourth line 
\draw (0,7.5)node{$\times$};
\draw (-.5,7.5)node{$x^{k_3}$};
\begin{pgfonlayer}{bg} 
\fill[color=white!60!Cerulean, draw=Cerulean] (0,7.5) rectangle (.4,7.9);
\end{pgfonlayer}
\path[postaction=decorate] (1.5,7.7) -- (4.5,7.7);
%% Fifth line 
\path[postaction=decorate] (0,8.4) -- (3,8.4);
\end{tikzpicture}
\caption{\label{figuArray}The (blue) closed square with bottom left corners $x^k$ and side length $\ell_k$ is the support of $ w_{h_k,\ell_k}(\cdot-x^k)$. 
}
\end{figure}
Let us show that $u$ has the desired properties. First, by construction $\supp u\subset [0,1]\times[0,L_\oo]$ so that  $u$ is compactly supported.  Next,~from the first identity of~\eqref{eq:rescale},
 \be\label{mupasbornee}
 |\mu[u]|(\Om)=|\pt_1\pt_2 u|(\Om) = 4\sum_{k\geq1} h_k\stackrel{\eqref{condetah}}=\oo.
 \ee
Let us now establish that $\F_{1+\te}( u)$ is finite. We emphasize that 
 \[
 \supp  w_{h_k,\ell_k}(\cdot - x^k)=[x_1^k,x_1^k+\ell_k]\times[x_2^k,x_2^k+\ell_k].
 \]
so that, from~\eqref{condxk}, the functions $w_{h_k,\ell_k}(\cdot,x^k)$ have disjoint supports and for $j\neq k$,
 \be\label{distsupports}
 d\lt(\supp w_{h_j,\ell_j}(\cdot - x^j),\supp w_{h_k,\ell_k}(\cdot - x^k)\rt)\ge(2-\sqrt2) \ell_k.
 \ee
By~\eqref{condetah2} for $\eps>0$ small enough, there exists an integer $k_\eps$ such that 
 \be\label{lkeps}
(2-\sqrt2)c\,\ell_{k_\eps} \le (2-\sqrt2)\ell_{k_\eps+1}<\eps\leq(2-\sqrt2)\ell_{k_\eps}.
 \ee
 Using $\supp\rho_\eps\subset B_\eps$ and~\eqref{distsupports}, we can write
  \be  \label{splitEeps}
  \E_{\eps,1+\te}(u) =\sum_{j=1}^{k_\eps} \E_{\eps,1+\te}(w_{h_j,\ell_j})+ \E_{\eps,1+\te}\lt(u_\eps\rt),
 \ee
 where we note%=\sum_{j=1}^{k_\eps} \E_{\eps,1+\te}(w_{h_j,\ell_j})+ \E_{\eps,1+\te}\lt(\sum_{j>k_\eps} w_{h_j,\ell_j}(\cdot-x^j)\rt)\\
 \[
u_\eps(x):=\sum_{j>k_\eps } w_{h_j,\ell_j}(x-x^j).
 \]
Let us first bound the remaining term $\E_{\eps,1+\te}(u_\eps)$ in~\eqref{splitEeps}. We notice that  $w_{h_j,\ell_j}(\cdot,x^j)$ is Lipschitz continuous with Lipschitz constant $\|\nb w_{h_j,\ell_j}\|_\oo=h_j/\ell_j\leq h_{k_\eps}/\ell_{k_\eps}$ for $j\geq k_\eps$. Since these functions have disjoint supports,
we conclude that their sum $u_\eps$ is Lipschitz continuous with
\[
\|\nb u_\eps\|_\oo\le h_{k_\eps}/\ell_{k_\eps}.
\]
Using $|D u_\eps(x,z_l)|\leq (h_{k_\eps}/\ell_{k_\eps}) |z|$ and $|D u_\eps(x,z_1)||D u_\eps(x,z_2)|=0$ if the three points $x$, $x+z_1$, $x+z_2$ lie in the complement of $\supp u_\eps$, we compute, 
\begin{align*}
\E_{\eps,1+\te}(u_\eps)&= \int_{B_\eps}\dfrac{\rho_\eps(z)}{|z|^{1+\te}}\int_{\R^2}|D u_\eps(x,z_1)|^{\te_1}|D u_\eps(x,z_2)|^{\te_2}\, dx\, dz \\
&\le \lt(\dfrac{h_{k_\eps}}{\ell_{k_\eps}}\rt)^\te \int_{B_\eps}\dfrac{\rho_\eps(z)}{|z|}|A_\eps(z)|\, dz,
\end{align*}
where $A_\eps(z)$ is the set of points $x\in \R^2$ such that at least one of the three points $\{x,x+z_1,x+z_2\}$ belongs to $\supp u_\eps$. We have 
\[
|A_\eps(z)|\le 3 \sum_{j\ge k_\eps} \ell_j^2\,.
\]
This leads to the estimate
\be\label{Eepsueps}
\E_{\eps,1+\te}(u_\eps)
\,\les\, \dfrac{h_{k_\eps}^\te}{\ell_{k_\eps}^\te\, \eps}  \sum_{j\ge k_\eps} \ell_j^2
\, \stackrel{\eqref{lkeps}}\les\,  \dfrac{h_{k_\eps}^\te}{\ell_{k_\eps}^{1+\te}}  \sum_{j\ge k_\eps} \ell_j^2
\, \stackrel{\eqref{condetah}}\longto \,0\quad\text{ as }\eps\dw0.
\ee
We now pass to the limit in the terms $\E_{\eps,1+\te}(w_{h_j,\ell_j})$ for $j\leq k_\eps$. We have 
\[
\E_{\eps,1+\te}(w_{h_j,\ell_j})=h_j^\te \ell_j^{1-\te} \E_{\eps/\ell_j,1+\te}(w).
\]
As in the proof of Proposition~\ref{rem:optimal} (i), we have 
\[
\E_{\eps/\ell_j,1+\te}(w)\,\up\,\F_{1+\te}( w)\quad\text{as }\eps\dw0.
\]
Summing over $1\leq j\leq k_\eps$, and sending $\eps\dw 0$, we get by monotone convergence theorem,
\[
\lim_{\eps\dw0}\sum_{j=1}^{k_\eps} \E_{\eps,1+\te}(w_{h_j,\ell_j})=\F_{1+\te}( w)\sum_{j\geq 1}h_j^\te \ell_j^{1-\te}.
\]
Combining this  together with~\eqref{splitEeps} and~\eqref{Eepsueps}, we obtain  $\F_{1+\te}(u)=\F_{1+\te}(w)\sum_{j}h_j^\te \ell_j^{1-\te}$ and by~\eqref{condetah} we conclude that $\F_{1+\te}(u)$ is finite whereas from~\eqref{mupasbornee}, $\mu[u]$ is not a finite measure.
\end{proof}

%\cite{White1999}\cite{Fleming66}\cite{Federer}\cite{Brezis2002}

%\bibliographystyle{alpha}
%\bibliography{BibStripes}
%\end{document}
%
% \section{Examples and counter-examples}
% {\red \`a faire compl\`etement}
% Dimension 1+1. $\un_E$, $min(x_1,x_2)$, optimality, Density Result in Dacorogna. Mettre le cas $u\in C^{1}$, avec $\tilde p=\te$ dans le cas $(b)$ et dire que si $u\in C^1$ tq $\partial_1 u=0$ ou $\partial_2 u=0$ en tous points alors $u=u_1$ ou $u=u_2$ donc si $u$ n est pas $1d$ et l energie est finie $p\le \tilde p$. 
%\end{document}
\section{The zero energy case}\label{S2}
In order to present the main ideas of the proof of Theorem~\ref{theointro:zero}, we start by considering the simplest possible setting. 
We restrict ourselves to $n_1=n_2=1$, $\te_1+\te_2=1$ and  work on the torus $\Om=\T=(\R/\Z)^2$ to avoid boundary effects (in particular, $\T^\eps=\T$). In this periodic setting, we need to distinguish the ambient manifold $\T$ from the space of tangent vectors $X=\R^2$, we define:
\[
\Om_1=\R/\Z\times\{0\},\quad \Om_2=\{0\}\times\R/\Z,\quad  X_1=\R\times\{0\},\quad X_2=\{0\}\times\R.
\]
With this notation, the definitions of $\U(\T)$ and of the energy are unchanged. 
\begin{proposition}\label{prop:main2d}
 Let $\te_1$, $\te_2>0$ be such that $\te_1+\te_2=1$ and let $u\in L(\T)$ be such that 
 \be\label{hypgoodE}
 \liminf_{\eps\dw 0}\int_{\R^2} \rho_\eps(z)\int_{\T} \frac{|Du(x,z_1)|^{\te_1} |Du(x,z_2)|^{\te_2}}{|z|^2} \,dx \,dz=0,
 \ee
then $u\in \U(\T)$.

\end{proposition}
\begin{proof}
As noticed in Remark~\ref{remarkLinfty}, we may assume without loss of generality that  $u\in L^\oo(\Om)$.\\
 {\it Step 1.} In this first step we prove that~\eqref{hypgoodE} allows us to find  a sequence $z^k=\eps_k (\sigma_1+\sigma_2)$ with $\sigma_1\in X_1\sm\{0\}$, $\sigma_2\in X_2\sm\{0\}$ and $\eps_k\dw 0$ such that 
 \be\label{goodzq}
  \lim_{k\up \oo} \int_{\T} \frac{q(x,z^k)}{\eps_k^2} \,dx=0, 
 \ee
 where we have set
 \be\label{def:q}
  q(x,z):=\lt(|Du(x+ z_2,z_1)|^{\te_1} +|Du(x,z_1)|^{\te_1}\rt)\lt(|Du(x+z_1,z_2)|^{\te_2}+|Du(x,z_2)|^{\te_2}\rt).
 \ee
Let us { digress slightly and} first  derive a consequence of~\eqref{goodzq}. For this, first notice that 
\begin{eqnarray*}
D u(x+z_1,z_2)-D u(x,z_2)&=&u(x+z_1+z_2)- u(x+z_2)-u(x+z_1)+u(x)\\
&=&D u(x+z_2,z_1)-D u(x,z_1),
\end{eqnarray*}
 so that by the triangle inequality we have
\begin{eqnarray*}
|D u(x+z_1,z_2)-D u(x,z_2)| &\leq & |Du(x,z_1)|+|D u(x+z_2,z_1)|\\
|D u(x+z_1,z_2)-D u(x,z_2)| &\leq & |Du(x,z_2)|+ |Du(x+z_1,z_2)|.
\end{eqnarray*}
Hence,
\be\label{triangle}
  |D u(x+z_1,z_2)-D u(x,z_2)|^{\te_1}\, \les |Du(x,z_1)|^{\te_1}+|D u(x+z_2,z_1)|^{\te_1},
 \ee
 and similarly for $\te_2$ so that  since $\te_1+\te_2=1$,
\begin{multline}\label{D1D2te=1}
 |D[Du(\cdot,z_2)(x,z_1)|= |Du(x+z_1,z_2)-Du(x,z_2)|  \\
 \les\lt(|Du(x+ z_2,z_1)|^{\te_1} +|Du(x,z_1)|^{\te_1}\rt)\lt(|Du(x+z_1,z_2)|^{\te_2}+|Du(x,z_2)|^{\te_2}\rt).
\end{multline}
Therefore, a consequence of ~\eqref{goodzq} would be the maybe more suggestive
\be\label{goodz}
  \lim_{k\up \oo} \int_{\T} \frac{|D[Du(\cdot,z_2^k)](x,z_1^k)|}{\eps_k^2} \,dx=0. 
 \ee
Observe that if $u$ were smooth{ , the integrand would converge to $|\sigma_1|\,|\sigma_2|\,|\pt_1\pt_2 u(x)|$ and~\eqref{goodz} } would directly imply $\pt_1 \pt_2 u=0$.\medskip

We now prove~\eqref{goodzq}. Making the change of variable $\tilde x=x+z_1$, $\tilde z=-z_1+z_2$ in~\eqref{hypgoodE}, we have
\[
 \liminf_{\eps\dw 0} \int_{\R^2} \rho_\eps(z)\int_{\T} \frac{|Du(\tilde x,\tilde z_1)|^{\te_1} |Du(\tilde x+\tilde z_1,\tilde z_2)|^{\te_2}}{|\tilde z|^2} d\tilde x d \tilde z=0. 
\]
Similarly we get 
\[
 \liminf_{\eps\dw 0}\int_{\R^2} \rho_\eps(z)\int_{\T} \frac{|Du(x+ z_2,z_1)|^{\te_1} |Du(x,z_2)|^{\te_2}}{|z|^2} \,dx \,dz=0
\]
and 
\[
 \liminf_{\eps\dw 0}\int_{\R^2} \rho_\eps(z)\int_{\T} \frac{|Du(x+ z_2,z_1)|^{\te_1} |Du(x+z_1,z_2)|^{\te_2}}{|z|^2} \,dx \,dz=0.
\]
Summing these { limits and~\eqref{hypgoodE}} we obtain,
\be\label{q2d}
 \liminf_{\eps\dw 0}\int_{\R^2} \rho_\eps(z)\int_{\T}\frac{q(x,z)}{|z|^2} \,dx \,dz\stackrel{\eqref{def:q}}{=} 0,
\ee
with $q$ given by~\eqref{def:q}. 
%\begin{multline*}\label{q2d}
% \liminf_{\eps\dw 0}\int_{\R^2} \rho_\eps(z)\int_{\T}\frac{q(x,z)}{|z|^2} \,dx \,dz\stackrel{\eqref{def:q}}{=}
% \\\liminf_{\eps\dw 0}\int_{\R^2} \rho_\eps(z)\int_{\T}\frac{\lt(|Du(x+ z_2,z_1)|^{\te_1} +|Du(x,z_1)|^{\te_1}\rt)\lt(|Du(x+z_1,z_2)|^{\te_2}+|Du(x,z_2)|^{\te_2}\rt)}{|z|^2} \,dx \,dz\\
%  =0.
%\end{multline*}
Using polar coordinates and the change of variables $r=\eps s$, we obtain 
\[
 \liminf_{\eps\dw 0} \int_{\partial B_1} \int_0^{\infty} \lt[\int_{\T} \frac{q(x,s\eps \sigma)}{s^2 \eps^2} \,dx\rt] s \rho(s) \,ds d\h^1(\sigma)=0.
\]
Applying Fatou Lemma and then Markov inequality, we may find for $l=1,2$, $\hat{\sigma}_l\in X_l\backslash\{0\}$ with $\frac{1}{4}\le |\hat{\sigma}_l|^2\le \frac{3}{4}$ and $s\in (0,1)$ such that setting $\sigma:=s(\hat{\sigma}_1+\hat{\sigma}_2)$, 
\[
 \liminf_{\eps\dw 0}\int_{\T} \frac{q(x,\eps \sigma)}{\eps^2} \,dx=0.
\]
Passing to a subsequence and noting $z^k:= \eps_k \sigma$, we get \eqref{goodzq}.\\

\medskip
{\it Step 2.} Let us show that~\eqref{goodzq} implies  $\partial_1 \partial_2 u =0$ in $\mathcal{D}'(\T)$. Let $\vhi\in C^\oo(\T)$. For every $x\in\T$, 
\begin{align*}
\pt_1\pt_2\vhi(x)&=\pm \dfrac1{|\sigma_1|\,|\sigma_2|} \pt_{\sigma_1}\pt_{\sigma_2}\vhi(x)=
\pm\dfrac1{|\sigma_1|\,|\sigma_2|}\lim_{k\up\oo } \dfrac{D[D \vhi(\cdot,-z_1^k)](x,-z_2^k)}{\eps_k^2}\,,
\end{align*}
with uniform convergence in $\T$. Multiplying by $u$, using the dominated convergence theorem and two discrete integration by parts, 
we compute
\begin{align}\nonumber
 \lt|\int_{\T} \partial_1\partial_2 \vhi u\, dx\rt|&=\dfrac1{|\sigma_1|\,|\sigma_2|}\, \lim_{k\up \oo} \,\dfrac1{\eps_k^2} \,\lt|\int_{\T} D[D \vhi(\cdot,-z_1^k)](x,-z_2^k) u(x) \,dx\rt|\\
 &=\dfrac1{|\sigma_1|\,|\sigma_2|}\,\lim_{k\up \oo} \, \dfrac1{\eps_k^2} \,  \lt|\int_{\T} \vhi D[D u(\cdot,z_1^k)](x,z_2^k) \,dx\rt|
 \label{thiscomputation}
 \\
 &\le\, \dfrac{\|\vhi\|_\oo}{|\sigma_1|\,|\sigma_2|}\, \lim_{k\up \oo}\int_{\T} \, \dfrac{|D[Du(\cdot,z_1^k)](x,z_2^k)|}{\eps_k^2} \,dx\stackrel{\eqref{goodz}}{=}0.
 \nonumber
\end{align}
Therefore, $\pt_1\pt_2u=0$ in the sense of distributions in $\T$.
%For this we introduce for $z\neq 0$ the discrete derivative
%\be\label{def:delta}
% \Delta u(x,z):= \frac{Du(x,z)}{|z|}
%\ee
%and notice that for smooth $u$ and $i= 1, 2$,  $\Delta u(x,z_i)$ is an approximation of $\partial_{i} u$ so that $\Delta[\Delta u(\cdot,z_2)](x,z_1)$ is an approximation of $\partial_1\partial_2 u$. For $\vhi\in C^\oo(\T)$, since~\eqref{goodz} can be rephrased as
%\be\label{goodzdelta}
% \lim_{k\up \oo} \int_{\T} |\Delta[\Delta u(\cdot,z_2^k)](x,z_1^k)|=0,
%\ee
%we have by Dominated Convergence Theorem
%\begin{align*}
% \lt|\int_{\T} \partial_1\partial_2 \vhi u\rt|&=\lim_{k\up \oo} \lt|\int_{\T} \Delta[\Delta \vhi(\cdot,-z_1^k)](x,-z_2^k) u(x) \,dx\rt|\\
% &=\lim_{k\up \oo}\lt|\int_{\T} \vhi \Delta[\Delta u(\cdot,z_1^k)](x,z_2^k) \,dx\rt|\\
% &\le \|\vhi\|_\oo \lim_{k\up \oo}\int_{\T}  |\Delta[\Delta u(\cdot,z_1^k)](x,z_2^k) \,dx|\stackrel{\eqref{goodzdelta}}{=}0,
%\end{align*}
%so that $\partial_1\partial_2 u=0$.\\

\medskip
{\it Step 3.} Integrating the relation $\partial_1\partial_2 u=0$ we find that $u(x)=u_1(x_1)+u_2(x_2)$ with $u_1$ and $u_2$ periodic functions on $\R$. Moreover, since $u$ is bounded, so are $u_1$ and $u_2$. Let us finally prove that $\pt_1 u_1\equiv 0$ or $\pt_2 u_2\equiv 0$.
From~\eqref{goodzq}, we have for some $\sigma_1\in X_1\sm \{0\}$, $\sigma_2\in X_2\sm\{0\}$ and a sequence $\eps_k\dw 0$, 
\begin{multline*}
 0=\lim_{k\up \oo} \int_{\T}\frac{|Du_1(x, \eps_k \sigma_1)|^{\te_1} |Du_2(x,\eps_k \sigma_2)|^{\te_2}}{\eps_k^2} \,dx\\
 =\lim_{k\up \oo} \lt(\int_0^1 \frac{|Du_1(x_1,\eps_k \sigma_1)|^{\te_1}}{\eps_k} \,dx_1\rt)\lt(\int_0^1 \frac{|Du_1(x_2,\eps_k \sigma_2)|^{\te_1}}{\eps_k} \,dx_2\rt),  
\end{multline*}
so that up to extraction  for $l=1$ or $l=2$, there holds
\[
 \lim_{k\up \oo} \int_0^1 \frac{|Du_l(x_l,\eps_k \sigma_l)|^{\te_l}}{\eps_k} \,dx_l=0.
\]
In particular, since $u_l$ is bounded and $0<\te_l<1$, 
\[
 \lim_{k\up \oo} \int_0^1 \frac{|Du_l(x_l,\eps_k \sigma_l)|}{\eps_k} \,dx_l \le \|u_l\|_\oo^{1-\te_l} \lim_{k\up \oo} \int_0^1 \frac{|Du_l(x_l,\eps_k \sigma_l)|^{\te_l}}{\eps_k} \,dx_l=0.
\]
Arguing as in~\eqref{thiscomputation}, we obtain that  $\pt_l u_l\equiv 0$ in the sense of distributions in $\R/\Z$. We conclude that $u\in \U(\T)$ which ends the proof of the proposition.%from which we may apply \cite[Th. 3.1]{MMS} as in the proof of \cite[Th. 4.1]{MMS} to obtain that $u_1$ is constant.
\end{proof}

We turn to the proof of Theorem~\ref{theointro:zero}, which extends Proposition~\ref{prop:main2d} in several directions by  considering general space dimensions and   powers $\te_1,$ $\te_2$.%(recall~\eqref{def:E0}).\\

Let us recall some notation. For $\Omega=\Omega_1+\Omega_2$ with $\Omega_1\subset X_1$ and $\Omega_2\subset X_2$ and a function $u\in L^1_{loc}(\Om)$, the matrix valued distribution $\mu[u]$ is defined as
\[
\mu=\mu[u]:=\nb_{1}\nb_{2} u= (\pt_{x_1^i}   \pt_{x_2^j}  u )_{\substack 1\leq i\leq n_1,\\1\leq j\leq n_2},
\]
where  $\dim X_1=n_1$ and $\dim X_2=n_2$. For $\te_1$, $\te_2>0$, we also recall  that the critical exponent $P(\te_1,\te_2)$ has been introduced in Definition~\ref{def:p}
%\be\label{P}
%p(\te_1,\te_2)\, :=\, 
%\left\{ \begin{array}{clr}
%2& \text{ if $\te\leq1$,}&(a)\smallskip\\ 
%1+\te& \text{ if $\te\ge 1$ and $\min(\te_1,\te_2)\leq 1$,}&(b)\smallskip\\
%\min(\te_1,\te_2)+\te& \text{ if  $\te_1,\te_2>1$.}&(c)
%\end{array}\right.
%\ee
and that when $p=P(\te_1,\te_2)$, we  simply write $\F( u)$ for $\F_{p}(u)$. 
\begin{remark}\label{rem:te=1}
We will  use the following inequality to reduce the case $\te<1$ to the case $\te=1$. Assume that $\te<1$ and let us note  $\te'_l:=\te_l+(1-\te)/2$ for $l\in\{1,2\}$. We have $\te':=\te'_1+\te'_2=1$ and for $u\in L^\oo(\Om)$ and  $x,z$ such that $x,x+z\in \Om$, we have 
\[
 |D u(x,z_1)|^{\te_1' } |Du(x,z_2)|^{\te_2'}
 \le 2^{1-\te} \|u\|_\oo^{1-\te}  |D u(x,z_1)|^{\te_1} |Du(x,z_2)|^{\te_2},
\]
so that 
\[
\F_{2}^{\te_1',\te_2'}(u)\le 2^{1-\te} \|u\|_\oo^{1-\te} \F_{2}^{\te_1,\te_2}(u).
\]
\end{remark}
We now prove that the energy $\F( u)$ controls the cross derivatives $\mu[u]=\nabla_1\nabla_2u$.
\begin{proposition}[Proposition~\ref{propquant:intro}]\label{propquant}
Let $u\in L^1_{loc}(\Om)$. We have the following estimates, for every $\vhi\in C^{\oo}_c(\Om,\R^{n_1\times n_2})$,
 \be
 \label{quanta}
  \langle\mu[u], \vhi\rangle\les 
  \begin{cases}
 \quad \qquad% \F( u) \|u\|_{\oo}^{1-\te} \|\vhi\|_{\oo} \qquad &\text{ if $\te=1$ or $\te<1$ and  $u\in L^\oo(\Om)$}, \smallskip\\
  \F( u)\|\vhi\|_{\oo} \qquad &\text{ if }\te=1\text{ or } [\te<1\text{ and  }\|u\|_\oo\leq 1], \smallskip\\
 \qquad \F( u)^{1/\te}\|\nb_1 \vhi\|_1^{1-1/\te}\|\vhi\|_{\oo}^{1/\te} \qquad &\text{ if }\theta\geq1 \text{ with } \te_1\le 1, \smallskip\\
\qquad  \F( u)^{1/\te}\|\nb_1 \vhi\|_1^{\te_2/\te}\|\vhi\|_{\oo}^{\te_1/\te} \qquad& \text{ if } 1\leq \te_1\leq  \te_2.
 \end{cases}
\ee
\end{proposition}
Because of the applications we have in mind in the last part of the paper (and in \cite{GM2}), it will actually be more convenient to derive Proposition~\ref{propquant} as a consequence of Lemma~\ref{lem:goodzk} and Lemma~\ref{lem:measure} below.

We may now state a first lemma which is the extension of Step 1 in the proof of Proposition~\ref{prop:main2d} to the more general setting. We recall that we defined $q(x,z)$ in~\eqref{def:q} as
\[
 q(x,z)=\lt(|Du(x+ z_2,z_1)|^{\te_1} +|Du(x,z_1)|^{\te_1}\rt)\lt(|Du(x+z_1,z_2)|^{\te_2}+|Du(x,z_2)|^{\te_2}\rt).
\]
We also recall that  $(e_1,\cdots, e_{n_1})$ and $(f_1,\cdots, f_{n_2})$ denote orthonormal bases of $X_1$ and $X_2$.
\begin{lemma}\label{lem:goodzk}
 If $u$ is such that $\F_{p}(u)<\infty$, then there exist  sequences $0< r_k\le \eps_k$ tending to $0$, two numbers $\lambda_1,\lambda_2\in (1/2,\sqrt{3}/2)$  and two rotations $R_l\in \SO(X_l)$ such that 
\be\label{eq:goodqk}
 \limsup_{k\up \oo}\sum_{1\leq i \leq n_1,1\leq j\leq n_2} \int_{\Om^{\eps_k}} \frac{q(x,  r_k(\lambda_1R_1e_i+\lambda_2R_2f_j))}{ r_k^p} \,dx 
 \les \F_{p}(u).
\ee
 
 %bases $(v_1,\cdots,v_{n_1})$ and $(w_1,\cdots,w_{n_2})$ of $X_1$ and $X_2$ with 
% \[
% \lambda_1:=|v_1|=\cdots=|v_{n_1}|,\qquad \lambda_2:=|w_1|=\cdots=|w_{n_2}|,\qquad 1/4\le\lambda_1^2,\lambda_2^2\le3/4,
% \]
%such that
% \be\label{eq:goodqk}
%  \limsup_{k\up \oo}\sum_{1\leq i \leq n_1,1\leq j\leq n_2} \int_{\Om^{\lambda \eps_k}} \frac{q(x,  \eps_k(v_i+w_j))}{ \eps_k^p} \,dx 
%  \les \F_{p}(u).
% \ee
%In particular in case $(a)$ of~\eqref{def:p} i.e. for $\te\le 1$ and $p=P(\te_1,\te_2)=2$ this implies (recall~\eqref{triangle}),
%\be\label{eq:goodD2}
% \limsup_{k\up \oo} \sum_{1\leq i \leq n_1,1\leq j\leq n_2} \int_{\Om^{\lambda \eps_k}} \frac{|D[D u(\cdot,  \eps_k w_j)](x, \eps_k v_i)|^\te}{ \eps_k^2} \,dx \les\F( u).
%\ee

\end{lemma}
\begin{proof}
% Let us first point out that in case $(a)$ of~\eqref{def:p}, inequality~\eqref{eq:goodD2} follows from~\eqref{eq:goodqk} by the exact same argument as for~\eqref{goodz},
%and thus we just need to show~\eqref{eq:goodqk}.
For $l=1,2$, let $\nu_l$ denotes the normalized Haar measure on $\SO(X_l)$ and let $\mathbf{S}(X_l)$ denote the unit sphere of $X_l$. We also set $\nu:=\nu_1\otimes \nu_2$. We first claim that for every $g\in L^1(\R^n)$, there holds 
\begin{multline}\label{moyenneintg}
 \int_{\R^n} g(z)\, dz\\
=\dfrac1{n_1n_2}\int_{\SO(X_1)\times \SO(X_2)}  \int_{\R^n} \lt[\sum_{1\leq i\le n_1,\,1\le j\le n_2}  g(|z_1| R_1 e_i+|z_2| R_2 f_j)\rt]\,dz \, d \nu(R_1,R_2).
\end{multline}
To prove \eqref{moyenneintg}, let us notice first that for   every function $v\in L^1(\mathbf{S}(X_1))$, the integral
$\int_{\SO(X_1)} v(R_1 e) d\nu_1(R_1)$ does not depend on the choice of  $e\in \mathbf{S}(X_1) $ and therefore 
\begin{align*}
 \int_{\SO(X_1)} v(R_1 e) d\nu_1(R_1)&=\frac{1}{\h^{n_1-1}(\mathbf{S}(X_1))}\int_{\mathbf{S}(X_1)}\int_{\SO(X_1)} v(R_1 \sigma) d\nu_1(R_1) d\h^{n_1-1}(\sigma)\\
 &= \frac{1}{\h^{n_1-1}(\mathbf{S}(X_1))}\int_{\SO(X_1)} \int_{\mathbf{S}(X_1)}v(R_1 \sigma)  d\h^{n_1-1}(\sigma)d\nu_1(R_1)\\
 &\stackrel{\sigma'=R_1 \sigma}{=} \frac{1}{\h^{n_1-1}(\mathbf{S}(X_1))}\int_{\SO(X_1)} \lt(\int_{\mathbf{S}(X_1)}v( \sigma')  d\h^{n_1-1}(\sigma')\rt)d\nu_1(R_1)\\
 &=\frac{1}{\h^{n_1-1}(\mathbf{S}(X_1))}\int_{\mathbf{S}(X_1)}v( \sigma)  d\h^{n_1-1}(\sigma).
\end{align*}

% Let us note $N_j= (n_l-1)n_l/2$ the dimension of $\SO(X_l)$ for $j=1,2$. Let $e\in X_1$, for $v\in L^1(\mathbf{S}(X_1))$, using the co-area formula, we have
% \begin{align*}
% \int_{\SO(X_1)} v(R_1 e)\,d\h^{N_1}(R_1)&=\int_{\mathbf{S}(X_1)} \h^{N_1-n_1+1}\lt(\{R_1\in \SO(X_1):R_1e=\sigma_1\}\rt)\,v(\sigma_1)\,d\h^{n_1-1}(\sigma_1)\\
% &=\h^{N_1-n_1+1}\lt(\SO(\R^{n_1-1})\rt)\, \int_{\mathbf{S}(X_1)} v(\sigma_1)\,d\h^{n_1-1}(\sigma_1).
% \end{align*}
For $g_1\in L^1(X_1)$, we deduce from the above formula and of a decomposition in polar coordinates that  
\begin{align*}
 \int_{X_1} g_1&= \int_0^{\infty} \h^{n_1-1}(\mathbf{S}(X_1)) \lt\{\int_{\SO(X_1)} g_1(rR_1 e)\,d\nu_1(R_1)\rt\} r^{n_1-1}\, dr\\
 &\stackrel{r=|z_1|}{=}  \int_{X_1} \lt\{\int_{\SO(X_1)} g_1(|z_1|R_1 e)\,d\nu_1(R_1)\rt\} \, dz_1\\
 &\stackrel{\text{Fubini}}{=} \int_{\SO(X_1)}  \lt(  \int_{X_1} g_1(|z_1|R_1 e)\, dz_1\rt) d\nu_1(R_1).
\end{align*}
Combining this with the analog formula for $f\in \mathbf{S}(X_2)$, $g_2\in L^1(X_2)$ and using Fubini, we have for $g\in L^1(\R^n)$,
\[
\int_{\R^n} g(z)\, dz=\int_{\SO(X_1)\times \SO(X_2)} \lt[ \int_{\R^n}  g(|z_1| R_1 e+|z_2| R_2 f)\,dz \rt]
\, d\nu(R_1,R_2).
\]
Taking $e=e_i$ and $f=f_j$ in the above identity and summing over $1\leq i\le n_1$, $1\le j\le n_2$, we obtain \eqref{moyenneintg}.\\

Define 
\[
g(z):=\dfrac{\rho_\eps(z)}{|z|^p}\int_{\Om^{2\eps}}q(x,z)\, dx,
\]
so that arguing as for \eqref{q2d} we have 
\[
 \int_{\R^n} g(z) \,dz\les \E_{\eps,p}(u).
\]

 Using \eqref{moyenneintg}, polar coordinates $z=r\sigma$, with $r>0$, $\sigma\in \pt B_1$ and the  change of variables $r=\eps s$, we obtain
\[
\int_{\SO(X_1)\times \SO(X_2)} 
\int_{\pt B_1}\int_0^{\oo}  
			Q_\eps(s,\sigma,R_1,R_2)
			s^{n-1}\rho(s)\, ds \, d\h^{n-1}(\sigma) \, d\nu(R_1,R_2)
			 \les\E_{\eps,p}(u),
\]
where we note 
\[
Q_\eps(s,\sigma,R_1,R_2):=\sum_{1\leq i\le n_1,\,1\le j\le n_2} 
			\int_{\Om^{2\eps}}\dfrac{q(x,s \eps |\sigma_1| R_1 e_i+s \eps |\sigma_2| R_2 f_j)}{s^p\,\eps^p}\,
			dx.
\] 

Passing to the infimum limit in $\eps$, using Fatou Lemma and then Markov inequality, we find that  there exist $R_1\in\SO(X_1)$, $R_2\in \SO(X_2)$, $\sigma\in \pt B_1$ with $1/2\le|\sigma_1|,|\sigma_2|\le\sqrt{3}/2$ and { $s\in [c,1)$,  where $c>0$ only depends on the kernel $\rho$}, such that
\[
\liminf_{\eps\dw 0} \sum_{1\leq i\le n_1,\,1\le j\le n_2} 
			\int_{\Om^{2\eps}}\dfrac{q(x,s \eps |\sigma_1| R_1 e_i+s \eps |\sigma_2| R_2 f_j)}{s^p\,\eps^p}\,
			dx
%\sum_{1\leq i\le n_1,\,1\le j\le n_2} \int_{\Om^{\eps}}\dfrac{q(x,s\eps |\sigma_1| R_1 e_i+s \eps |\sigma_2| R_2 f_j)}{s^p\,\eps^p}\, dx
\les\F_{p}(u).
\]
%Eventually, we set $v_i:=|\sigma_1|R_1 e_i$, for $i=1,\cdots,n_1$ and $w_j:=|\sigma_2|R_2 f_j$, for $j=1,\cdots,n_2$, we have 
%\[
%\liminf_{\eps\dw 0}\sum_{1\leq i\le n_1,\,1\le j\le n_2} 
%			\int_{\Om^{\eps}}\dfrac{q(x,s\eps  v_i+s\eps w_j)}{s^p\,\eps^p}\, dx
%\les\F_{p}(u).
%\]
Extracting a subsequence $\eps'_k$ realizing the liminf we find
\[\limsup_{k\up \infty} \sum_{1\leq i\le n_1,\,1\le j\le n_2} 
			\int_{\Om^{2\eps'_k}}\dfrac{q(x, (s\eps'_k) [|\sigma_1| R_1 e_i+ |\sigma_2| R_2 f_j])}{\,(s\eps'_k)^p}\,
			dx\les\F_{p}(u).
\]
Noting $\lambda_l:= |\sigma_l|$ for $l=1,2$, $\eps_k:=2\eps'_k$ and $r_k:= s\eps'_k$ we conclude the proof of \eqref{eq:goodqk}.
\end{proof}

 \begin{remark}\label{remark:normalisation}
With the notation of the lemma, we define the map $A\in \GL(\R^n)$ by  $Az:=\lambda_1 R_1 z_1+ \lambda_2 R_2 z_2$.
Making the change of variables $x=A \hat {x}$  and $\hat{u}(\hat x)=u(A \hat{x})$ in~\eqref{eq:goodqk}, defining $\hat{\eps}_k:=\max(\lambda_1^{-1},\lambda_2^{-1})\eps_k$ 
and observing that $[A^{-1}(\Om)]^{\hat{\eps}_k}\subset A^{-1}(\Om^{\eps_k})$,
we find
\[
\limsup_{k\up \oo}\sum_{1\leq i \leq n_1,1\leq j\leq n_2} \int_{[A^{-1}(\Om)]^{\hat{\eps}_k}} \dfrac{\hat q(x,r_k(e_i+f_j))}{ r_k^p}
\,dx \les \F_{p}(u)\sim \F_{p}(\hat{u}), 
\]
with the definition of $\hat q$ modeled on the definition of $q$:
\[
\hat q(x,z):=\lt(|D\hat u(x+ z_2,z_1)|^{\te_1} +|D\hat u(x,z_1)|^{\te_1}\rt)\lt(|D\hat u(x+z_1,z_2)|^{\te_2}+|D\hat u(x,z_2)|^{\te_2}\rt).
\]

%  \begin{multline*}
%  \lim_{k\up \oo} \int_{A^{-1}\Omega{ r_k}} \frac{\lt(|D\hat{u}(\hat{x}+ f_j,e_i)|^{\te_1} +|D\hat{u}(\hat x,e_i)|^{\te_1}\rt)\lt(|D\hat{u}(\hat{x}+e_i,f_j)|^{\te_2}+|D\hat{u}(\hat{x},f_j)|^{\te_2}\rt)}{ r_k^p} d\hat x\\
%  \les \F_{p}(u)\sim \F_{p}(\hat{u}). 
%  \end{multline*}
Since, $X_1$ and $X_2$ are stable by $A$, we have $u\in \U(\Om) \Longleftrightarrow \hat u\in \U(A^{-1}\Om)$. Therefore, up to this change of variables, 
we may always assume that~\eqref{eq:goodqk} holds true with $\lambda_l R_l=\Id_{X_l}$ for $l=1,2$.
 \end{remark}

\begin{lemma}\label{lem:measure}
Let $u\in L^1_{loc}(\Omega)$ and let  $(\eps_k)$, $(r_k)$  be two sequences with $0<r_k\leq \eps_k\dw0$. For $i\in \{1,\cdots,n_1\}$ and $j\in \{1,\cdots,n_2\}$, we define
\[
 \mathcal{F}_{i,j}(u):=\liminf_{k\up \oo} \int_{\Om^{\eps_k}} \frac{q(x, r_k(e_i+f_j))}{r_k^{P(\te_1,\te_2)}} \,dx.
\]
(i) If $\te=1$ or $[\te<1$ and  $\|u\|_\oo\leq 1]$ (in both cases, $P(\te_1,\te_2)=2$), we have for $i\in\{1,\cdots,n_1\}$, $j\in \{1,\cdots,n_2\}$,
\be\label{eq:goodD2}
\G_{i,j}(u):= \liminf_{k\up \oo}  \int_{\Om^{\eps_k}} \frac{|D[D u(\cdot,  r_k f_j)](x, r_k e_i)|}{ r_k^2} \,dx \les \mathcal{F}_{i,j}(u).
 \ee
 (ii) Noting $\mu=\mu[u]$, we have for  $i\in\{1,\cdots,n_1\}$, $j\in \{1,\cdots,n_2\}$ and every $\vhi\in C^\oo_c(\Om)$,
\be
 \label{quantaintegral}
  \langle\mu_{i,j}, \vhi\rangle\ \les\  
  \begin{cases}
~\qquad  \G_{i,j}(u)  \|\vhi\|_\oo \  &\text{ in any cases,}\quad\hfill(a) \smallskip\\
 \ (\mathcal{F}_{i,j}(u))^{1/\te}\|\nb_1 \vhi\|_1^{1-1/\te}\|\vhi\|_{\oo}^{1/\te} \  &\text{ if $\theta\geq 1$  and } \te_1\le 1, \quad (b)\smallskip\\
  \ (\mathcal{F}_{i,j}(u))^{1/\te}\|\nb_1 \vhi\|_1^{\te_2/\te}\|\vhi\|_{\oo}^{\te_1/\te} \ & \text{ if }1<\te_1\le \te_2.\quad\hfill(c)
  \end{cases}
\ee
\end{lemma}
\begin{proof}
 \textit{Part (i).} Let us first consider the case $\te= 1$, so that $P(\te_1,\te_2)=2$. Let $i\in\{1,\cdots,n_1\}$, $j\in \{1,\cdots,n_2\}$ and let $k\ge1$ and $x\in\Om^{\eps_k}$.  We note $z^k=r_k(e_i+f_j)$. As in~\eqref{D1D2te=1} we have
 \begin{multline*}
  |D[Du(\cdot,z_2^k)](x,z_1^k)| \\  \les\lt(|Du(x+ z_2^k,z_1^k)|^{\te_1} +|Du(x,z_1^k)|^{\te_1}\rt)\lt(|Du(x+z_1^k,z_2^k)|^{\te_2}+|Du(x,z_2^k)|^{\te_2}\rt)
  =q(x,z^k).
 \end{multline*}
Integrating in $x$ over $\Om^{\eps_k}$ we get~\eqref{eq:goodD2}.\\
In the case $\te<1$ and $u\in L^\oo(\Om)$, noting $\te'_l=\te_l/\te$ for $l=1,2$, by triangle inequality, we have for every $x\in \Om^\eps$, $z\in B_\eps$ and $l\in\{1,2\}$,
\begin{align*}
 |Du(x,z_l)|^{\te'_l} & =|u(x+z_l) -u(x)|^{\te_l'} = |u(x+z_l)-u(x)|^{(1-\te)\te_l/\te}\,|Du(x,z_l)|^{\te_l} \\
 & \leq 2^{(1-\te)\te_l/\te}\|u\|_\oo^{(1-\te)\te_l/\te}|Du(x,z_l)|^{\te_l}.
\end{align*}
Since $\te'_1+\te'_2=1$, the case $[\te<1$ and $\|u\|_\oo\leq 1]$ follows form this inequality and from the case $\te=1$ applied to the pair $(\te'_1,\te'_2)$.\medskip

\textit{Part (ii).} Let $u\in L^1_{loc}(\Om)$, $i\in\{1,\cdots,n_1\}$, $j\in\{1,\cdots,n_2\}$ and $\vhi\in C^\oo_c(\Om)$. For $k\geq 1$, we note $z^k=r_k(e_i+f_j)$. We treat the three cases (a), (b), (c) of~\eqref{quantaintegral} separately.

\textit{Case (a).}  We proceed exactly as in~\eqref{thiscomputation}.  Using discrete integration by parts, we compute 
\begin{align*}
\lt<\mu_{i,j},\vhi\rt>&\,= \int_{\Om} \partial_{e_i}\partial_{f_j} \vhi\, u\, dx = \lim_{k\up \oo} \,\dfrac1{r_k^2} \,\lt|\int_{\Om} D[D \vhi(\cdot,-z_1^k)](x,-z_2^k) u(x) \,dx\rt|\\
 &=\lim_{k\up \oo} \, \dfrac1{r_k^2} \,  \lt|\int_{\Om} \vhi D[D u(\cdot,z_1^k)](x,z_2^k) \,dx\rt|
 \label{thiscomputation}
 \\
 &\le\,\|\vhi\|_\oo\, \liminf_{k\up \oo}\int_{\Om^{\eps_k}} \, \dfrac{|D[Du(\cdot,z_1^k)](x,z_2^k)|}{r_k^2} \,dx = \,\|\vhi\|_\oo\,\G_{i,j}(u).
\end{align*}
This proves the claim.
\medskip

\textit{Case (c).} Let us show that this case follows from case (b). Let us assume that  $1\le \te_1\le \te_2$. By H\"older inequality,
 \begin{multline*}
  \int_{\Omega^{\eps_k}}\frac{\lt(|Du(x+ z^k_2,z^k_1)| +|Du(x,z^k_1)|\rt)\lt(|Du(x+z^k_1,z^k_2)|^{{\te_2}/{\te_1}}+|Du(x,z^k_2)|^{{\te_2}/{\te_1}}\rt)}{ r_k^{2+{\te_2}/{\te_1}}} \,dx \\
  \les |\Om|^{\frac{\te_1-1}{\te_1}}\lt(\int_{\Omega^{\eps_k}}\frac{\lt(|Du(x+ z^k_2,z^k_1)|^{\te_1} +|Du(x,z^k_1)|^{\te_1}\rt)\lt(|Du(x+z^k_1,z^k_2)|^{\te_2}+|Du(x,z^k_2)|^{\te_2}\rt)}{ r_k^{2\te_1+\te_2}} \,dx\rt)^{\frac1{\te_1}}.
 \end{multline*}
Applying~(\ref{quantaintegral}.b) with $\te'_1=1$, $\te'_2={\te_2}/{\te_1}$ and $P(\te'_1,\te'_2)=1+\te_1'+\te'_2=2+{\te_2}/{\te_1}$  yields~(\ref{quantaintegral}.c).\medskip

\textit{Case (b).}  From now on we assume $ \te_1\leq 1$, $\te \geq1$  and $p =P(\te_1,\te_2)= 1+\te$. Let $x\in \Om$ and $z\in\R^n$ be such that $x\in \Om^{|z|}$.
%Let us recall the notation (see~\eqref{def:q})
%\[
% q(x,z)=\lt(|Du(x+ z_2,z_1)|^{\te_1} +|Du(x,z_1)|^{\te_1}\rt)\lt(|Du(x+z_1,z_2)|^{\te_2}+|Du(x,z_2)|^{\te_2}\rt).
%\]
Arguing as in~\eqref{triangle}, we obtain by triangle inequality,
\[
  |D[Du(\cdot,z_2)](x,z_1)|^{\te_1}\les |Du(x+ z_2,z_1)|^{\te_1} +|Du(x,z_1)|^{\te_1},
\]
and\footnote{Remark that this is where we  used  the hypothesis $\te_1\le1$ : in case (c) ($\te_1,\te_2>1$), this inequality fails and the present method breaks down.} 
\[
 |D[Du(\cdot,z_2)](x,z_1)|^{1-\te_1}\les \lt( |Du(x+z_1,z_2)|^{\te_2}+|Du(x,z_2)|^{\te_2}\rt)^{(1-\te_1)/\te_2}.
 \]
  Taking the product of the last two inequalities with $(|Du(x+z_1,z_2)|^{\te_2}+|Du(x,z_2)|^{\te_2})^{(\te-1)/\te_2}$, we get 
\[
(|Du(x+z_1,z_2)|^{\te_2}+|Du(x,z_2)|^{\te_2})^{\frac{\te-1}{\te_2}} \,|D[Du(\cdot,z_2)](x,z_1)|\les q(x,z).
\]
We use this estimate in the form
\be
\label{nonlinestim}
(|Du(x+z_1,z_2)|^{\te-1}+|Du(x,z_2)|^{\te-1})\,|D[Du(\cdot,z_2)](x,z_1)|\les q(x,z).
\ee
%  \lt( |Du(x+z_1,z_2)|^{\te-1}+|Du(x,z_2)|^{\te-1}\rt)\les  |Du(x+z_1,z_2)|^{\te_2}+|Du(x,z_2)|^{\te_2},
%\]
We now set $\phi(s):= |s|^{\te-1} s$ for $s\in\R$. Using the estimate 
\[
 \lt| |s|^{\te-1} s- |t|^{\te-1} t\rt|\les |s-t| \lt(|s|^{\te-1} +|t|^{\te-1} \rt),
\]
and~\eqref{nonlinestim}, we  have
\begin{align*}
|D[ \phi(Du(\cdot,z_2))](x,z_1)|& \les|D[Du(\cdot,z_2)](x,z_1)| \lt( |Du(x+z_1,z_2)|+|Du(x,z_2)|\rt)^{\te-1}\\
&\les q(x,z).
\end{align*}
Applying this estimate with $z=z^k$, integrating in $x\in\Om^{\eps_k}$ and sending $k$ to $\oo$, we have: 
\be\label{estimphi}
 \liminf_{k \to \infty} \int_{\Om^{\eps_k}} \frac{|D[ \phi(Du(\cdot,z^k_2))](x,z^k_1)|}{r_k^{1+\te}}\,dx \les \mathcal{F}_{i,j}(u).
\ee
For $v:\Om\to\R$, $e\in \mathbf{S}(X_1)\cup \mathbf{S}(X_2)$, $r>0$ and $x\in \Om^{r}$, we introduce the discrete derivative, 
\[D_{r} v(x,e):=\frac{D v(x,r e)}{r}=\dfrac{v(x+re)-v(x)}r.\]
  With this notation,~\eqref{estimphi} rewrites as
\be\label{estimphidelta}
 \liminf_{k \to \infty} \int_{\Om^{\eps_k}} |D_{r_k}[ \phi(D_{r_k} u(\cdot, f_j))](x,e_i)| \,dx \les \mathcal{F}_{i,j}(u).
\ee
For smooth functions $u$, this inequality would provide a control on the $L^1$-norm of the function $\pt_{e_i}\phi(\pt_{f_j}u)$. Here, we only assume $u\in L^1_{loc}(\Om)$ and it is difficult to give a meaning to the nonlinear term $\phi(\pt_{f_j}u)$. 
For this reason, we linearize $\phi$ away from $0$. 
For $\eta>0$, we introduce the function $\phi_\eta$ given by 
\[
 \phi_\eta(s):=\begin{cases}
               \frac{|s|^{\te-1} s}{\eta^{\te-1}}  &\text{if } |s|\le \eta,\\
               s &\text{if } |s|\ge \eta,
              \end{cases}
\]
so that $\phi_\eta $ is an odd, Lipschitz continuous function satisfying $0\le \phi_\eta'\leq \eta^{1-\te} \phi'$ on $\R_+$. As a consequence,
\[
|D_{r_k}[\phi_\eta(v)](x, e_i)|\leq \eta^{1-\te}|D_{r_k}[\phi(v)](x, e_i)|. 
\]
Thus, from~\eqref{estimphidelta} we have,
\be\label{estimphieta}
  \liminf_{k \to \infty} \int_{\Om^{\eps_k}} |D_{r_k}[ \phi_\eta(D_{r_k} u(\cdot, f_j))](x, e_i)| \,dx \les\eta^{1-\te} \mathcal{F}_{i,j}(u).
\ee
For $\vhi\in C^\oo_c(\Om,\R)$,
% \[
%D_{r_k}[D_{r_k} \vhi(\cdot,-r_k e_i)](x,-r_k f_j) \ \stackrel{k\up\oo}\longto\ \frac{\partial^2 \vhi}{\partial e_i \partial f_j}(x)\quad \text{uniformly for }x\in\Om.
% \]
using the dominated convergence theorem and a discrete integration by parts, we compute
 \begin{align}
  \langle \mu_{i,j}, \vhi  \rangle&=\lim_{k\up\oo}\int_{\Om^{\eps_k}} u \frac{\partial^2 \vhi}{\partial e_i \partial f_j}  \,dx\, =\,\lim_{k\up \oo} \int_{\Om^{\eps_k}} u(x)D_{r_k}[D_{r_k} \vhi(\cdot,- e_i)](x,- f_j)  \,dx\nonumber\\
  &=\lim_{k\up \oo} \int_{\Om^{\eps_k}} D_{r_k} u(x, f_j) D_{r_k} \vhi(x,- e_i)  \,dx.\label{muvhifg}
 \end{align}
 We introduce the decomposition $D_{r_k} u(x, f_j)= \phi_\eta (D_{r_k} u(x, f_j))-\chi_\eta(D_{r_k} u(x, f_j))$ where $\chi_\eta(s):= \phi_\eta(s)-s$ satisfies
\[
 \|\chi_\eta\|_\oo=\sup_{|s|\le \eta} \lt|s\lt(1-\frac{|s|^{\te-1}}{\eta^{\te-1}}\rt)\rt|\le \eta.
\]
For $k$ large enough (so that $\supp\vhi\subset\Omega^{\eps_k}$),  we have
\begin{multline*}
 \lt|\int_{\Omega^{\eps_k}} D_{r_k} u(x,f_j) \, D_{r_k} \vhi(x,-e_i)  \,dx\rt|\\ 
 \le \lt|\int_{\Omega^{\eps_k}} \phi_\eta(D_{r_k} u(x,f_j)) \, D_{r_k} \vhi(x,-e_i)  \,dx\rt|
 +\lt|\int_{\Omega^{\eps_k}} \chi_\eta(D_{r_k} u(x,f_j)) \, D_{r_k} \vhi(x,-e_i)  \,dx\rt|\\
% &\le \lt|\int_{\Omega^{\eps_k}} D_{r_k}[\phi_\eta(D_{r_k} u(\cdot,f_j))](x,e_i)  \vhi(x)  \,dx \rt|\\
% &\qquad \qquad +\int_{\Omega^{\eps_k}} \lt|\chi_\eta(D_{r_k} u(x,f_j))\rt| \lt|D_{r_k} \vhi(x,-e_i)\rt|  \,dx\\
 \le \|\vhi\|_\oo  \int_{\Omega^{\eps_k}} |D_{r_k}[\phi_\eta(D_{r_k} u(\cdot,f_j))](x,e_i)|  \,dx + \eta \|\pt_{e_i}\vhi\|_{L^1},%\int_{\Omega^{\eps_k}}  \lt|D_{r_k} \vhi(x,-e_i)\rt|  \,dx
\end{multline*}
where we used a discrete integration by parts to treat the  first term and the bound $\|\chi_\eta\|_\oo\leq\eta$ for the second term.
Using~\eqref{muvhifg} and~\eqref{estimphieta}, we obtain,
\[
 \langle \mu_{i,j}, \vhi  \rangle\les \|\vhi\|_\oo\eta^{1-\te} \mathcal{F}_{i,j}(u)+ \eta \|\nabla_1 \vhi\|_1.
\]
Eventually, optimizing in $\eta$ by choosing $\eta^\theta=  \mathcal{F}_{i,j}(u) \|\vhi\|_\oo  / \|\nabla_1\vhi\|_1$, we get~(\ref{quantaintegral}.b).
\end{proof}
\begin{proof}[Proof of Proposition~\ref{propquant}]
For $\te> 1$ the proposition corresponds to (\ref{quantaintegral}.b) and~(\ref{quantaintegral}.c). For $\te\leq 1$, the proposition follows from~\eqref{eq:goodD2} and~(\ref{quantaintegral}.a).
\end{proof}

We can now show that if $\F(u)=0$ then $u$ depends only on the variables in $X_1$ or only on the variables in $X_2$.

\begin{theorem}[Theorem~\ref{theointro:zero}]\label{theo:zero}
If $u\in L(\Om)$ is such that $\F(u)=0$, then $u\in \U(\Om)$. 
\end{theorem}
\begin{proof}
Let $u\in L(\Om)$ with $\F(u)=0$. As in the proof of Proposition~\ref{prop:main2d}, we may assume that $u$ is bounded. Applying Lemma~\ref{lem:goodzk} (see also Remark~\ref{remark:normalisation}) we find sequences  $0< r_k\leq \eps_k\dw 0$ such that up to a change of coordinates,  
\begin{equation}\label{goodzzero}%\be\label{goodzrig}
 \lim_{k\up \oo} \sum_{i,j} \int_{\Om^{\eps_k}}\frac{q(x,r_k(e_i+f_j))}{r_k^{P(\te_1,\te_2)}}\,dx 
 %\le\lim_{k\up \oo}\sum_{i,j}  \int_{\Om^{\eps_k}} \frac{q(x,r_k(e_i+f_j))}{r_k^{1+\te}}\,dx
 =0.   
\end{equation}%\ee
By Lemma~\ref{lem:measure}, we get
 \begin{equation*}%\label{diffzero}
\mu:=\mu[u]=  \nb_1\nb_2 u=0\qquad\text{in }\mc{D}'(\Om).
 \end{equation*} 
Integrating twice this identity, since $\Om_1$ and $\Om_2$ are connected, there exist two distributions $u_l\in\mc{D}'(\Om_l)$, $l\in\{1,2\}$ such that $u=u_1\otimes \un_{\Om_2} +\un_{\Om_1}\otimes u_2$ in $\mc{D}'(\Om)$. Let $\vhi_2\in C^\oo_c(\Om_2,\R_+)$ with $\int \vhi_2=1$. Using test functions of the form $\vhi(x)=\vhi_1(x_1)\vhi_2(x_2)$, we have, since $u$ is bounded,
 \[
 \lt<u_1,\vhi_1\rt> \le \lt( \|u\|_\oo +\lt|\lt<u_2,\vhi_2\rt>\rt|  \rt)\|\vhi_1\|_{L^1} \qquad \text{for every }\vhi_1\in C^\oo_c(\Om_1).
 \]
 We deduce that $u_1\in L^\oo(\Om_1)$ and similarly $u_2\in L^\oo(\Om_2)$. In conclusion, we have two functions $u_l\in L^\oo(\Om_l)$, $l\in\{1,2\}$ with $u(x)=u_1(x_1)+u_2(x_2)$ for $x\in \Om$.\\
%We now claim that there exist functions $u_1$ and $u_2$ such that 
%\be\label{split}
%u(x)\,=\, u_1(x_1)+u_2(x_2).
%\ee
%Up to convolving~\eqref{diffzero}, we can assume that $u$ is smooth. For $x_2\in \Om_2$, let $v_{x_2}=\nb_2 u(\cdot+x_2)$. Since $\nb_1 v_{x_2}=0$ and since $\Om_{1}$
%is connected, there exists $\xi: \Om_{2}\to \R^{n_2}$ such that $v_{x_2}(x_1)=\xi(x_2)$. Since $\Om_{2}$ is simply connected and $\nabla_2\times \xi=0$, there exists $u_2 :\Om_{2}\to \R$ such that $\xi=\nb_2 u_2$.
%Then, for every $x_1\in \Om_{1}$, letting $u_{x_1}=u(x_1+\cdot)$, 
%$\nb_2( u_{x_1}-u_2)=0$ so that there exists $u_1    :\Om_{1}\to \R$ such that $u_{x_1}-u_2=u_1$, which proves~\eqref{split}.
%Since $u$ is bounded, we see that $u_1$ and $u_2$ are also bounded.\\
 Since $q(x,r_k(e_i+f_j))\ge |D u(x,r_k e_i)|^{\te_1} |Du (x,r_k f_j)|^{\te_2} $, \eqref{goodzzero} implies that   
\[%\be\label{goodzrig}
 \lim_{k\up \oo} \sum_{i,j} \int_{\Om^{\eps_k}}\frac{|D u(x,r_k e_i)|^{\te_1} |Du (x,r_k f_j)|^{\te_2}}{r_k^{P(\te_1,\te_2)}}\,dx 
 %\le\lim_{k\up \oo}\sum_{i,j}  \int_{\Om^{\eps_k}} \frac{q(x,r_k(e_i+f_j))}{r_k^{1+\te}}\,dx
 =0.   
\]%\ee
It is easy to check from the formula of $P(\te_1,\te_2)$ of Definition~\ref{def:p} that $\max(1,\te_1)+\max(1,\te_2)\le P(\te_1,\te_2)$ for $\te_1,\te_2>0$. Moreover, for $l=1,2$, we have $|Du|^{\max(1,\te_l)}\les|Du|^{\te_l} \|u\|_\oo^{(1-\te_l)_+}$. These inequalities lead to
\[
 \lim_{k\up \oo}  \lt(\sum_{i} \int_{\Om_1^{\eps_k}} \frac{|D u_1(x_1,r_k e_i)|^{\max(1,\te_1)}}{r_k^{\max(1,\te_1)}} \,dx_1\rt)\lt(\sum_{j}  \int_{\Om_2^{\eps_k}} \frac{|D u_2(x_2,r_k f_j)|^{\max(1,\te_2)}}{r_k^{\max(1,\te_2)}} \,dx_2\rt)=0.
\]
Therefore, up to a subsequence, either 
\begin{multline*}
  \lim_{k\up \oo} \sum_{i}  \int_{\Om_1^{\eps_k}} \frac{|D u_1(x_1,r_k e_i)|^{\max(1,\te_1)}}{r_k^{\max(1,\te_1)}} \,dx_1=0 \\
 \text{ or } \qquad  \lim_{k\up \oo}\sum_{j} \int_{\Om_2^{\eps_k}} \frac{|D u_2(x_2,r_k f_j)|^{\max(1,\te_2)}}{r_k^{\max(1,\te_2)}} \,dx_2=0.
\end{multline*}
Let us assume without loss of generality that the former holds. Using H\"older inequality, this yields, for $i\in\{1,\cdots,n_1\}$, 
\[
  \lim_{k\up \oo}  \int_{\Om_1^{\eps_k}} \frac{|D u_1(x_1,r_k e_i)|}{r_k} \,dx_1=0. 
  \]
We deduce that the distributions $\pt_{e_i}u_1$ vanish for $i\in\{1,\cdots,n_1\}$ and thus (since $\Om_1$ is connected)  $u_1$ is constant in $\Omega_1$. This  concludes  the proof of the theorem.
%This implies that either the first quantity vanishes for every $i\in [1, n_1]$ or the second quantity vanishes for every $j\in [1,n_2]$. Let us assume without loss of generality that the former holds, then  we may apply 
% \cite[Th. 3.1]{MMS} to conclude that $u_1$ is constant in $\Omega_1$.  
\end{proof}

\begin{corollary}[Corollary~\ref{CoroIntro:geneGRun}]\label{Coro:geneGRun}
 For every $p\ge P(\te_1,\te_2)$, $r>0$, and every measurable function $u$, if 
 \begin{equation}\label{hyp:GRun}
   \int_{B_r}\int_{\Om^r} \frac{|u(x+z_1)-u(x)|^{\te_1} |u(x+z_2)-u(x)|^{\te_2}}{|z|^{n+p}} dx dz <\infty,
 \end{equation}
then $u\in \U(\Om^r)$.
\end{corollary}
\begin{proof}[Proof of Corollary~\ref{Coro:geneGRun}]
 Let $\rho:=\frac{1}{|B_r\backslash B_{r/2}|} \chi_{B_r\backslash B_{r/2}}$. Then, letting $\eps_k:=2^{-k} r$, it is readily seen that~\eqref{hyp:GRun} implies that 
 \[
 \sum_{k\geq 0}\int_{\R^n} \rho_{\eps_k}(z)  \int_{(\Omega^r)^{\eps_k}}\frac{|u(x+z_1)-u(x)|^{\te_1} |u(x+z_2)-u(x)|^{\te_2}}{|z|^{p}} \,dx \,dz=0.
 \]
 In particular, $\F_{p}(u;\Om^r)=0$ and we get $u\in \U(\Om^r)$ from Theorem~\ref{theo:zero}.
\end{proof}
In the space of Lipschitz continuous functions, we know from Proposition~\ref{rem:optimal}~(i) that the critical exponent $p$ is larger than $1+\te$.
We deduce from Theorem~\ref{theo:zero} that, as soon as $\te_1\leq 1$ or $\te_2\le 1$, this critical exponent is indeed $1+\te$.
\begin{proposition}[Proposition~\ref{Lip:intro}]\label{propLip}
Assume that  $\min(\te_1,\te_2)\leq 1$. Then, for $u\in Lip(\Om)$,  $[\F_{1+\te}(u)=0] \Longrightarrow u\in \U(\Om)$. 
\end{proposition}
\begin{proof}
Let us assume without loss of generality that $\te_1 \leq 1$. Let $u\in Lip(\Om)$ with $\F_{1+\te}(u)=0$. If $\te\geq1$ then $P(\te_1,\te_2)= 1+\te$ and by Theorem~\ref{theo:zero}, we have $u\in \U(\Om)$. If $\te<1$, by Lipschitz continuity of $u$, we have,  for $x\in \Om$ and $z\in\R^n\bks\{0\}$ such that $x+z\in \Om$,
\begin{align*}
\dfrac{|Du(x,z_1)|\,|Du(x,z_2)|^{\te_2}}{|z|^{2+\te_2}} &\leq \|\nb_1 u\|_{\oo}^{1-\te_1} |z_1|^{1-\te_1}\dfrac{|Du(x,z_1)|^{\te_1}\,|Du(x,z_2)|^{\te_2}}{|z|^{2+\te_2}}\\
&\leq  \|\nb_1 u\|_{\oo}^{1-\te_1}\dfrac{|Du(x,z_1)|^{\te_1}\,|Du(x,z_2)|^{\te_2}}{|z|^{1+\te}}.
\end{align*}
This yields $\F_{2+\te_2}^{1,\te_2}(u)\leq  \|\nb_1 u\|_{\oo}^{1-\te_1} \F_{1+\te}^{\te_1,\te_2}(u)=0$. We notice that $P(1,\te_2)=2+\te_2$, so that applying Theorem~\ref{theo:zero} with $\te'_1=1$,  $\te'_2=\te_2$, we get again $u\in \U(\Om)$.
\end{proof}

\section{Quantitative control of the distance to $\U(\Om)$ in terms of $\F$}\label{S3}
 The aim of this section is to give a quantitative version of Theorem~\ref{theo:zero} by proving that for $u\in L^\oo(\Om)$, $\F( u)$ controls the distance to $\U(\Om)$ in a strong sense. 
In order to obtain such a strong control, we use the fact that $\mu=\nabla_1 \nabla_2 u$ is a measure and thus restrict ourselves to the case $(a)$ of~\eqref{Pintro}, i.e. $\te_1+\te_2\le 1$.  We start by investigating 
the two dimensional case where the proof is simpler and the result stronger. 
%\textcolor{red}{check compatibility of statement/proof with the $SBV_\te$ case below}
\begin{theorem}[Theorem~\ref{BVestim1dintro}]\label{BVestim1d}
Assume that $n_1=n_2=1$ and $\te=\te_1+\te_2\le 1$. Then, for every $u\in L^\oo(\Om)$ with $\|u\|_\oo\leq 1$ and $\F( u)<\oo$, there exists $\bar u\in \U(\Om)$  such that $u-\bar u\in BV(\Om)\cap L^\oo(\Om)$ with the estimate
\be\label{quantrig2d}
\|u-\bar u\|_\oo+|\nb[u-\bar u]|(\Om)\les \F( u)+ \F( u)^{\frac{1}{2}}.
\ee
% \begin{multline}\label{quantrig2d}
% \|u-\bar u\|_\oo+|\nb[u-\bar u]|(\Om)\les \lt( \|u\|_\oo^{1-\te}\F( u)+\|u\|_\oo^{\te_1+(1-\te)(1+\frac{1}{2}(\te_2-\te_1))}\F( u)^{\te_2+\frac{1}{2}(1-\te)}\rt.\\
% \lt.+\|u\|_\oo^{\te_2+(1-\te)(1+\frac{1}{2}(\te_1-\te_2))}\F( u)^{\te_1+\frac{1}{2}(1-\te)}\rt).
% \end{multline}
\end{theorem}

\begin{proof} 
To set notation, we assume that $\Om=(0,\ell_1)\times(0,\ell_2)$, so that $(0,0)$ is the bottom left corner of $\Om$. By Proposition~\ref{propquant}, $\mu:=\partial_1\partial_2 u$ is a finite Radon measure with 
$|\mu|(\Om)\les \F( u)$. For $x\in\Om$, we set $w(x):=\mu((0,x_1]\times(0,x_2])$ and notice that 
\be
\label{bornew}
\|w\|_\oo + |\nb w|(\Om)\, \le 3 |\mu|(\Om)\,\les\, \F( u).
\ee
Since $\partial_1\partial_2 (u-w)=0$, arguing as in the proof of Theorem~\ref{theo:zero}, we find two functions $u_l\in L^\oo((0,\ell_l))$, $l\in\{1,2\}$ such that  
\[
u(x)=u_1(x_1)+u_2(x_2)+w(x) \qquad\text{for }x\in \Om.
\]
Thanks to~\eqref{bornew},  we only need to control $u_1$ or $u_2$.  
Let $x_1$ and $z_1>0$ be such that $0<x_1<x_1+z_1<\ell_1$, we have for $x_2\in (0,\ell_2)$,
\[
Du(x,z_1)=Du_1(x_1,z_1) + D w(x,z_1)=Du_1(x_1,z_1) +\mu((x_1,x_1+z_1]\times (0,x_2]).
\]
In particular, for $x_2\in(0,\ell_2)$,
\be\label{g1}
|Du(x,z_1)|\, \geq\, \Big[|D u_1(x_1,z_1)|-|\mu|((x_1,x_1+z_1]\times (0,\ell_2))\Big]_+\,=:\,\psi_1(x_1,z_1).
\ee
Similarly, for $x_1\in(0,\ell_1)$ and $x_2$, $z_2>0$ such that $0<x_2<x_2+z_2<\ell_2$,
\be\label{g2}
|Du(x,z_2)|\, \geq\, \Big[|D u_2(x_2,z_2)|-|\mu|((0,\ell_1)\times(x_2,x_2+z_2]\Big]_+\,=:\,\psi_2(x_2,z_2).
\ee
Notice for later use that for $l\in\{1,2\}$,
\be\label{Linftygi}
|\psi_l(x_l,z_l)|\leq\ |D u(x,z_l)|\leq 2 \|u\|_\oo\le 2, \quad \text{for every }x,z\in X\text{ with }x,x+z\in\Om. 
\ee
Plugging inequalities~\eqref{g1} and~\eqref{g2} in the definition~\eqref{Eeps} of $\E_\eps(u)$, we get, 
\[
 \liminf_{\eps\dw 0} \int_{\R^2}\rho_\eps(z) \int_{(\eps,\ell_1-\eps)\times(\eps,\ell_2-\eps)}\dfrac{[\psi_1(x_1,z_1)]^{\te_1}\, [\psi_2(x,z_2)]^{\te_2}}{|z|^2}\,  \,dx \,d z \,\leq\, \F( u).
\]
Arguing as in the proof of Lemma~\ref{lem:goodzk},  we may select $\sigma_1+\sigma_2=\sigma\in \partial B_1$,  with $1/2\le |\sigma_l|\le\sqrt3/2 $ for $l=1, 2$ and two sequences  sequence $(r_k)$, $(\eps_k)$ with $0<r_k\le \eps_k\dw0$  such that letting $z^k:=r_k\sigma$, 
\[
\lim_{k\up \oo} \lt(\int_{\eps_k}^{\ell_1-\eps_k} \dfrac{[\psi_1(x_1,z_1^k)]^{\te_1}}{r_k} \,dx_1\rt) \lt(\int_{\eps_k}^{\ell_2-\eps_k} \dfrac{[\psi_2(x_2,z_2^k)]^{\te_2}}{r_k} \,dx_2\rt) \, \les \, \F( u).
\]
Extracting a further subsequence, there exists $l\in\{1,2\}$ such that 
\be\label{borneintgi}
\limsup_{k\up \oo} \int_{\eps_k}^{\ell_l-\eps_k} \dfrac{\lt[\psi_l(x_l,z_l^k)\rt]^{\te_l}}{r_k} \,dx_l\, \les \, \F( u)^{\frac{1}{2}}.
\ee
Without loss of generality, we assume  $l=1$. Using the definition of $\psi_1$ and Fubini, we write
\begin{align*}
\int_{\eps_k}^{\ell_1-\eps_k} |Du_1(x_1,z_1^k)| \,dx_1\,
&\le \, \int_{\eps_k}^{\ell_1-\eps_k} |\mu|((x_1,x_1+z_1^k]\times (0,\ell_2)) \,dx_1 +\int_{\eps_k}^{\ell_1-\eps_k} \psi_1(x_1,z_1^k) \,dx_1\\
&\le\, z_1^k\,|\mu|(\Om) + \lt(\sup_{(\eps_k,\ell_1-\eps_k)} \psi_1\lt(\cdot,z_1^k\rt)\rt)^{1-\te_1}\int_{\eps_k}^{\ell_1-\eps_k}\lt[\psi_1\lt(x_1,z_1^k\rt)\rt]^{\te_1} \,dx_1\\
&\stackrel{\eqref{Linftygi}}{\les}\, r_k\, \F( u) + \int_{\eps_k}^{\ell_1-\eps_k}\lt[\psi_1\lt(x_1,z_1^k\rt)\rt]^{\te_1} \,dx_1.
\end{align*}
Dividing by $r_k$ and letting $k\up \oo$, we deduce from~\eqref{borneintgi},% and the fact that  $\te=1$, 
\[%\be\label{conclusion:BV2d}
\limsup_{k\up \oo} \int_{\eps_k}^{\ell_1-\eps_k}  \dfrac{|Du_1(x_1,z_1^k)|}{r_k} \,dx_1\, 
\les\,  \F( u) +  \F( u)^{\frac{1}{2}}.
\]%\ee
We conclude that $u_1\in BV(0,\ell_1)$ with $|\partial_1u_1|(0,\ell_1)\les \F( u) + \F( u)^{\frac{1}{2}}$.  
This implies $\osc(u_1)\les  \,\F( u) +  \F( u)^{\frac{1}{2}}$, which together with~\eqref{bornew} concludes the proof of the theorem.
\end{proof}
%\begin{remark}
%In the case $\te<1$, we will prove later on that ~\eqref{quantrig2d} can be improved to a control of $u-\bar u$ in the stronger $SBV_\te(\Omega)$ norm (see Corollary~\ref{cor:te<1}).
%\end{remark}

We now turn to the higher dimensional case $n> 2$. In order to obtain the analog of~\eqref{quantrig2d}, we must define the higher dimensional counterpart of $w$. 
As will be clear from the proofs, the main requirements are that $\nabla_1\nabla_2 w=\mu[u]$ and that for almost every $x=x_1+x_2\in\Om$,
\be\label{average}
 \int_{\Omega_2} w(x_1+ y_2) dy_2 =0 \qquad\text{ and }\qquad \int_{\Omega_1} w(y_1+ x_2) dy_1 =0.
\ee
The unique function which satisfies these conditions is given by the formula:
\be\label{def:w}
 w(x):=u(x)-\mint_{\Om_2} u(x_1+y_2) \, dy_2 -\mint_{\Om_1} u(y_1+x_2) \, dy_1  +\mint_{\Om} u(y) \, dy.
 \ee
 For $x=x_1+x_2\in \Om$, we have the decomposition $u(x)=u_1(x_1)+u_2(x_2)+w(x)$  where $u_1$ and $u_2$ are explicitly given by:
\be\label{def:uj}
u_1(x_1):= \mint_{\Om_2} u(x_1+y_2)\, dy_2 -\dfrac12 \mint_\Om u dy,\qquad u_2(x_2):= \mint_{\Om_1} u(y_1+x_2)\, dy_1  -\dfrac12 \mint_\Om u dy.
\ee
We start by establishing some bounds on $u_1$, $u_2$ and $w$.
\begin{proposition}\label{prop:boundw}
 We assume that $\Om_1$ and $\Om_2$ are bounded extension domains. Let $u\in L^\oo(\Om)$ and $w$, $u_1$, $u_2$  be given by~\eqref{def:w}\eqref{def:uj}. 
 Then, $w\in L^\oo(\Om)$, $u_l\in L^\oo(\Om_l)$ for $l=1,2$ with $u(x)=u_1(x_1)+u_2(x_2)+w(x)$ and we have the estimates
 \be\label{boundwu1u2}
 \|w\|_\oo\le 4\|u\|_\oo,\qquad \|u_l\|_\oo\le (3/2)\|u\|_\oo\quad \text{ for } l=1,2.
 \ee
Moreover, if $\te\le 1$ and $\F( u)<\oo$ then $w\in BV(\Om)$ and denoting $\hat n:=\max(n_1,n_2)\geq 2$,
 %then denoting $\hat{n}:=\max(n_1,n_2)$ we have  $u\in BV(\Om)\cap L^{\hat{n}/(\hat{n}-1)}(\Om)$ with the estimate,
 \be\label{boundw}
 \|w\|_{L^{{\hat n}/({\hat n}-1)}(\Om)} + |\nb w|(\Om) \les \,\|u\|_\oo^{1-\te}\,\F( u).
 \ee
\end{proposition}

\begin{proof}
The estimates of~\eqref{boundwu1u2} follow from the definitions~\eqref{def:w},~\eqref{def:uj} and the triangle inequality. 
 
We turn to~\eqref{boundw}. By~\eqref{quanta} and the definition of $w$, $\nb_1\nb_2 w=\mu$ with $|\mu|(\Om)\les \F( u)\|u\|_\oo^{1-\te}$. It is thus enough to prove
 \be
 \label{boundw_bis}
   \|w\|_{L^{{\hat n}/({\hat n}-1)}(\Om)} + |\nb w|(\Om)\, \les\, |\mu|(\Om).
 \ee
By density ($\Om$ is a bounded extension domain) we may assume that  $u\in C^\oo(\ov\Om)$.  { For every $x_1\in \Om_1$, we have $W_{x_1}:=\nb_1 w(x_1+\cdot)\in W^{1,1}(\Om_2)$ with, 
\[
\int_{\Om_1} \|\nb_2 W_{x_1}\|_{L^1(\Om_2)}\, dx_1 =\int_{\Om_1}\int_{\Om_2} |\nb_2 \nb_1 w|  \,=  |\mu|(\Om).
\]
By~\eqref{average}, we have $\smint_{\Om_2}W_{x_1}\,=0$ for every $x_1\in \Om_1$ and by Poincar\'e-Wirtinger inequality in $W^{1,1}(\Om_2)$, this yields
\[
 \int_{\Om} |\nb_1w| \,dx= \int_{\Om_1} \|W_{x_1}\|_{L^1(\Om_2)}\, dx_1\les \int_{\Om_1} \|\nb_2 W_{x_1}\|_{L^1(\Om_2)}\, dx_1= |\mu|(\Om).
\]}
The analog bound on $\nb_2 w$ shows that 
\be\label{boundw_ter}
\|\nb w\|_{L^1(\Om)}\, \les\, |\mu|(\Om).
\ee
We finally establish the $L^{\frac{\hat{n}}{\hat{n}-1}}$ bound. Assume without loss of generality that $\hat{n}=n_1\ge n_2$.
We then have by~\eqref{average} and Poincar\'e-Wirtinger inequality in $W^{1,1}(\Om_1)$ that 
\[
\lt\| w \rt\|_{L^{\frac{n_1}{n_1-1}}(\Om)}\les \lt(\int_{\Om_2}\lt(\int_{\Om_1} |\nb_1 w| \,dx_1\rt)^{\frac{n_1}{n_1-1}} \,dx_2\rt)^{\frac{n_1-1}{n_1}}
=\,\lt\|\int_{\Om_1}|\nb_1 w(x_1+\cdot)|\, dx_1  \rt\|_{L^{\frac{n_1}{n_1-1}}(\Om_2)}.
 \]
 By  Minkowski inequality, this leads to 
\[
\lt\| w \rt\|_{L^{\frac{n_1}{n_1-1}}(\Om)}
\les  \int_{\Om_1} \lt\| \nb_1 w(x_1+\cdot )\rt\|_{L^{\frac{n_1}{n_1-1}}(\Om_2)} \, dx_1.
%\les  \int_{\Om_1} \lt\| \nb_1 w(x_1+\cdot )\rt\|_{L^q(\Om_2)} \, dx_1
 %\les  \int_{\Om_1}\lt(\int_{\Om_2} |\nb_1 w(x)|^{\frac{n_1}{n_1-1}}\, dx_2\rt)^{\frac{n_1-1}{n_1}}\, dx_1
% \,\les \,  \int_{\Om_1}\lt(\int_{\Om_2} |\nb_1 w(x)|^{\frac{n_2}{n_2-1}}\, dx_2\rt)^{\frac{n_2-1}{n_2}}\, dx_1,
\]
%we used the H\"older inequality and have set $q=n_2/(n_2-1)$ if $n_2\geq 2$ and $q=\infty$ if $q=1$ (in both cases $n_2\geq n_1/(n_1-1)$ and $n_2\leq n_1$ in the last inequality. 
We now apply the Poincar\'e-Wirtinger inequality to $\nb_1 w(x_1+\cdot)$ in $W^{1,1}(\Om_2)$ to obtain (since $n_2\leq n_1$),  
\[
\lt\| w \rt\|_{L^{\frac{n_1}{n_1-1}}(\Om)}
 \,\les \,  \int_{\Om_1}\lt\|\nb_2 \nb_1 w(x_1+\cdot)\rt\|_{L^1(\Om_2)}\, dx_1 = |\nb_1\nb_2 u|(\Om)=  |\mu|(\Om).
\]
Together with~\eqref{boundw_ter} this proves~\eqref{boundw_bis}.
\end{proof}
\begin{remark}
 We point out that the Poincar\'e-Wirtinger inequality gives a slightly stronger result than~\eqref{boundw}, namely that $\nb_1 w\in \mathcal{M}(\Om_1,BV(\Om_2))$ (and similarly for $\nb_2w$).
\end{remark}

We turn to the higher dimensional analog of Theorem~\ref{BVestim1d}. 
\begin{proposition}\label{thm:u-baru}
Assume that $\Om_1$ and $\Om_2$ are bounded extension domains and  that $\te\le1$. Let $u\in L^\oo(\Om)$ with $\|u\|_\oo\leq 1$ and $\F( u)<\oo$.  Using the notation $u=u_1+u_2+w$ of Proposition~\ref{prop:boundw}, there exist $l\in\{1,2\}$ and~$c\in\R$ such that,
\be \label{estimbaruj}
\|u_l-c\|_{L^{\frac{n_l}{n_l-1}}(\Om)}\les \,\F(u)^{\frac{1}{2}}.
\ee
\end{proposition}
%where $q_j=n_l/(n_l-1)$ if $n_l\ge 2$ and $q_j=\infty$ if $n_l=1$.
% Then, there exist positive  constants  $c$ and $C$ depending only on $n$ and $\Om$ such that if $\|u\|_{\oo}^{1-\te}\F (u)\le c$,  then there exists $\bar u\in G$, such that denoting $\hat n:=\max(n_1,n_2)$, 
%  \be\label{estimbaru}
%   \|u-\bar u\|_{L^{\frac{\hat{n}}{\hat{n}-1}}(\Om)} \le C \lt((\|u\|_{\oo}^{1-\te}\F (u))^{1/2} + \|u\|_{\oo}^{1-\te}\F (u)\rt) \lt(1+\|u\|_\oo\rt).
%  \ee
Combining this estimate,  H\"older inequality $\|u_l-c\|_{L^{\frac{\hat{n}}{\hat{n}-1}}(\Om)}\les \|u_l-c\|_{L^{\frac{n_l}{n_l-1}}(\Om_l)}$ and Proposition~\ref{prop:boundw},  we deduce the following result.
\begin{theorem}[Theorem~\ref{thm:LpestimIntro}]\label{coro:u-baru}
 Let $\Om_1$ and $\Om_2$ be bounded extension domains and    $\te\le1$. Let $u\in L^\oo(\Om)$ with $\|u\|_\oo\leq 1$ and $\F (u)<\oo$. Noting $\hat n:=\max(n_1,n_2)$, 
 there exists $\bar u\in L^\oo(\Om)\cap \U(\Om)$ such that, 
 \begin{equation*}%\label{estimbaru}
  \|u-\bar u\|_{L^{\frac{\hat{n}}{\hat{n}-1}}(\Om)} \les \, \F (u)+\F(u)^{\frac{1}{2}}
   \end{equation*}
% \be\label{estimbaru}
%  \|u-\bar u\|_{L^{\frac{\hat{n}}{\hat{n}-1}}(\Om)} \les \lt((\|u\|_{\oo}^{1-\te}\F (u)^{1/2} + \|u\|_{\oo}^{1-\te}\F (u)\rt) \lt(1+\|u\|_\oo\rt).
% \ee
\end{theorem}
\begin{proof}[Proof of Proposition~\ref{thm:u-baru}]
Since $\|u\|_\infty\le 1$, by \eqref{boundwu1u2} it is enough to establish \eqref{estimbaruj} under the additional assumption that 
\be\label{addassumpteta}
\F(u)\leq \eta,
\ee
for some $\eta>0$ only depending on $\Om$, $n_1$, $n_2$ and $\rho$.

{\it Step 1.} 
 Let $w$, $u_1$ and $u_2$ be given by~\eqref{def:w} and~\eqref{def:uj} so that $u(x)=u_1(x_1)+u_2(x_2)+w(x)$.   Let us first recall that by \eqref{quanta}, $\nb_1\nb_2 w=\nb_1\nb_2 u=\mu$
 is a measure with $|\nb_1\nb_2 w|(\Om)=|\mu|(\Om)\les \F( u)$. Let $(w_k)$ be a sequence of mollifications of $w$ with 
$ w_k\to w$ in $L^1(\Om)$  and almost everywhere, $ \nb w_k\to \nb w$  weakly star in $\mathcal{M}(\Om)$, $\|\nb w_k\|_{L^1(\Om)}\to|\nb w|(\Om)$
 and $\|\nb_1\nb_2 w_k\|_{L^1(\Om)}\to|\mu|(\Om)$
(which is possible since $\Om_1$ and $\Om_2$ are extension domains). For $k\ge 0$ and $x_1\in\Om_1$, we set 
\[
\psi_1^k(x_1):=\int_{\Om_2}|\nb_2 w_k|(x_1+y_2) dy_2.
\] 
We have $\| \psi_1^k \|_{L^1(\Om_1)}=\| \nb_2 w_k \|_{L^1(\Om)}$ and $\|\nb_1 \psi_1^k\|_{L^1(\Om)} \le \|\nb_1\nb_2 w_k\|_{L^1(\Om)}$. Hence, by~\eqref{boundw}
\[
\limsup_{k\up\oo}\ \| \psi_1^k \|_{L^1(\Om_1)}+ \|\nb_1 \psi_1^k\|_{L^1(\Om_1)} \,\les\, \F( u).
\]
Therefore, $\psi_1^k$ is bounded in $W^{1,1}(\Om_1)$ and  up to extraction, $\psi_1^k$ converges in $L^1(\Om_1)$ to some function $\psi_1\in BV(\Om_1)$ with 
\be\label{boundpsi1}
\|\psi_1\|_{L^1(\Om_1)} + |\nb_1 \psi_1|(\Om_1)\les \F( u).
\ee
Let then for $\lambda>0$,
\[
  \Om_{1,\lambda}:=\{x_1 \in \Om_1 :  \psi_1(x_1)\le \lambda\}.%=\{x_1 \in \Om_1 :  |\nb_2 w|(x_1,\cdot)({\Om_2})\le \lambda\}.
 \]
In the sequel we use that for a.e. $x_1\in \Om_{1,\lambda}$ and every $z_2\in X_2$ such that $|(\Om_2-z_2)\cap\Om_2|>0$, 
\begin{align}
 \int_{(\Om_2-z_2)\cap \Om_2} \frac{|D w|(x,z_2)}{|z_2|} \,dx_2&=\lim_{k\up \oo} \int_{(\Om_2-z_2)\cap \Om_2} \frac{|D w_k|(x,z_2)}{|z_2|} \,dx_2\nonumber\\
 &\le\limsup_{k\up \oo} \int_{(\Om_2-z_2)\cap \Om_2}\int_0^1 \frac{|\nabla w_k(x+t z_2)\cdot z_2|}{|z_2|} dt \,dx_2\nonumber \\ 
& \le \lim_{k\up \oo} \int_{\Om_2}|\nb_2 w_k(x)|  \,dx_2  =\psi_1(x_1)\le \lambda.\label{DvsGrad}
\end{align}
Now,  by co-area formula \cite[Theorem 3.40]{Am_Fu_Pal}  for almost every $\lambda>0$ the set $\Om_{1,\lambda}$  is of finite perimeter in $\Om_1$ and denoting by $\partial \Om_{1,\lambda}$ the measure-theoretic boundary of $\Om_{1,\lambda}$,
\[
\int_0^\oo \h^{n_1-1}(\partial \Om_{1,\lambda} \cap \Om_1) \, d\lambda\,=|\nb_1 \psi_1|(\Om_1)\ \stackrel{\eqref{boundpsi1}}\les\ \F( u). 
\]
We thus pick $\lambda_1>0$ such that
\be
\label{lbdprime}
\frac{\F(u)^{\frac{1}{2}}}{4}\leq \lambda_1\leq \frac{\F(u)^{\frac{1}{2}}}{2}\quad \text{and}\quad  \h^{n_1-1}(\partial \Om_{1,\lambda} \cap \Om_1)\les \F(u)^{\frac{1}{2}}.
\ee
We set $\widetilde\Om_1:=\Om_{1,{\lambda_1}}$ and  define $\om_1:= \Om_1\ \sm \widetilde\Om_1$ (so that $ \h^{n_1-1}(\partial \Om_{1,\lambda} \cap \Om_1)= \h^{n_1-1}(\partial \om_1 \cap \Om_1)$). We notice for later use that
\[
 \lambda_1 |\om_1|\, \leq \int_{\om_1} |\psi_1| \,dx_1\leq \|\psi_1\|_{L^1(\Om_1)}\ \stackrel{\eqref{boundpsi1}}\les\ \F( u),
 \]
Therefore,  since $\lambda_1\geq \F(u)^{\frac{1}{2}}/4$, { choosing $\eta$ small enough in~\eqref{addassumpteta},} we have, 
\be
\label{lbdprime2}
|\om_1|\les \F(u)^{\frac{1}{2}} \leq {|\Om_1|}/{2}.
\ee
We define $\psi_2$, $\lambda_2$, $\widetilde{\Om}_2$ and $\om_2$ similarly.

{\it Step 2.} Let $(r_k)$, $(\eps_k)$ with $0<r_k\leq \eps_k\dw 0$ be  given by Lemma~\ref{lem:goodzk} (recall Remark~\ref{remark:normalisation}) and such that for every $1\le i\le n_1$ and 
$1\le j\le n_2$,
\be\label{hyp:control}
 \limsup_{k\up \oo} \int_{\Omega^{\eps_k}} \frac{|Du(x,r_k e_i)|^{\te_1}|Du(x,r_k f_j)|^{\te_2}}{r_k^2} \,  \,dx\les \F( u).
\ee
Fix for the moment $k\ge1$, $1\le i\le n_1$ and $1\le j \le n_2$ and let $z^k:=r_k(e_i+f_j)$. We define 
\[
 A_1^k:=\lt\{x \in\Omega^{\eps_k} :  |D u|(x,z^k_1)\ge \frac{1}{2} |D u_1|(x_1,z^k_1)\rt\}
\]
and similarly for $A_2^k$. Now we write
\begin{multline}\label{prdcOmlbd}
\lt( \int_{\widetilde{\Om}_1\cap \Omega^{\eps_k}} \frac{|D u_1|(x_1,z^k_1)}{r_k} \,dx_1\rt)\lt(\int_{\widetilde{\Om}_2\cap \Omega^{\eps_k}} \frac{|D u_2|(x_2,z^k_2)}{r_k} \,dx_2\rt) \\
=\int_{(\widetilde{\Om}_1 \times \widetilde{\Om}_2)\cap \Omega^{\eps_k}} \frac{|D u_1|(x_1,z^k_1)|D u_2|(x_2,z^k_2)}{r_k^2}  \,dx,
\end{multline}
and we split the domain of integration as 
\[
(\widetilde{\Om}_1 \times \widetilde{\Om}_2)\cap \Omega^{\eps_k}\subset \lt[  A_1^k\cap  A_2^k\rt] 
\cup \lt[  \lt[(\Om_1\times \widetilde{\Om}_2)\cap \Omega^{\eps_k} \rt]\bks   A_1^k\rt]
 \cup \lt[ \lt[(\widetilde{\Om}_1\times \Om_2)\cap \Omega^{\eps_k}\rt] \bks  A_2^k\rt].
\] 
In the first subdomain, we use the inequalities $|D u_1|(x_1,z^k_1)\leq 2 |Du|(x,z^k_1)$ and $|D u_2|(x_2,z^k_2)\leq 2 |Du|(x,z^k_2)$ valid by definition for $x\in  A_1^k\cap  A_2^k$. Then, since $\|u\|_\oo\leq1$, we get
\begin{multline}
\label{GprimeGseconde}
\limsup_{k\up \oo} \int_{A_1^k\cap  A_2^k} \frac{|D u_1|(x_1,z^k_1)|D u_2|(x_2,z^k_2)}{r_k^2} \,dx\,\\
\leq \, 4 \limsup_{k\up \oo} \int_{\Omega^{\eps_k}} \frac{|D u|^{\te_1}(x,z_1^k)|D u|^{\te_2}(x,z_2^k)}{r_k^2} \,dx
\stackrel{\eqref{hyp:control}}{\les}  \F( u).
\end{multline}
In the second subdomain, we use the triangle inequality  to get  for $x\not\in A_1^k$,  
\[
 |Du_1(x_1,z_1^k)|  = |D u(x,z_1^k)-D w(x,z_1^k) | \ \stackrel{x\not\in A_1^k}\le\  \dfrac{|D u_1(x_1,z_1^k)|}2 + |D w(x,z_1^k)|,
\] 
so that for $x\not\in A_1^k$, there holds $ |Du_1(x_1,z_1^k)| \leq 2 |Dw|(x,z_1^k)$. This leads to
\begin{multline*}
\int_{((\Om_1\times \widetilde{\Om}_2)\cap \Omega^{\eps_k})\bks   A_1^k}\frac{|D u_1|(x_1,z^k_1)|D u_2|(x_2,z^k_2)}{r_k^2} \,dx\\
\leq\ 2 \int_{\widetilde{\Om}_2\cap \Om_2^{\eps_k}} \frac{|D u_2|(x_2,z^k_2)}{r_k}\lt( \int_{\Om_1^{\eps_k}}\frac{|Dw|(x,z_1^k)}{r_k} \,dx_1\rt)\,d x_2.
\end{multline*}
By definition of $\widetilde{\Om}_2$ and~\eqref{DvsGrad}, the inner integral is bounded by $\lambda_2\stackrel{\eqref{lbdprime}}{\leq} \F(u)^{\frac{1}{2}}/2$. Hence,
\begin{multline}
\label{OmscdmoinsGprime}
\limsup_{k\up \oo}\int_{((\Om_1\times \widetilde{\Om}_2)\cap \Omega^{\eps_k})\bks   A_1^k}\frac{|D u_1|(x_1,z_1^k)|D u_2|(x_2,z_2^k)}{r_k^2} \,dx\\
\leq  \F(u)^{\frac{1}{2}}\, \limsup_{k\up \oo}\int_{\widetilde{\Om}_2\cap \Om_2^{\eps_k}} \frac{|D u_2|(x_2,z_2^k)}{r_k} \,dx_2	.
\end{multline}
The third term is bounded similarly and we deduce from~\eqref{prdcOmlbd},\eqref{GprimeGseconde} and~\eqref{OmscdmoinsGprime},
\begin{multline*}
\lt[ \limsup_{k\up \oo}\int_{\widetilde{\Om}_1\cap \Om_1^{\eps_k}} \frac{|D u_1|(x_1,z_1^k)}{r_k} \,dx_1 -\F(u)^{\frac{1}{2}}\rt]_+
\\
\lt[\limsup_{k\up \oo}\int_{\widetilde{\Om}_2\cap \Om_2^{\eps_k}} \frac{|D u_2|(x_2,z_2^k)}{r_k} \,dx_2-\F(u)^{\frac{1}{2}}\rt]_+
\ \les  \F( u).
\end{multline*}
Recalling the notation $z_1^k=r_ke_i$, $z_2^k=r_kf_j$ and summing over $i\in\{1,\cdots,n_1\}$ and $j\in\{1,\cdots, n_2\}$, we see that up to extraction, we have 
\begin{align}\label{estimsumi}
&\text{For  }i\in\{1,\cdots,n_1\},&&  \limsup_{k\up \oo}\int_{\widetilde{\Om}_1\cap \Om_1^{\eps_k}} \frac{|D u_1|(x_1,r_k e_i)}{r_k} \,dx_1 
\,\les    \F(u)^{\frac{1}{2}},
\\\nonumber
&\text{or for }j\in\{1,\cdots,n_2\}, & & \limsup_{k\up \oo}\int_{\widetilde{\Om}_2\cap \Om_2^{\eps_k}} \frac{|D u_2|(x_2,r_k f_j)}{r_k} \,dx_2
\,\les   \F(u)^{\frac{1}{2}} .
\end{align}

{\it Step 3.} Let us assume without  loss of generality that the first possibility occurs and let us define the function 
\[
 \tilde u_1:= u_1 \quad \text{in } \widetilde{\Om}_1 \qquad \tilde u_1:=0 \quad \text{in } \om_1.
\]
We fix again $i\in \{1,\cdots,n_1\}$ and use the short-hand notation $z_1^k:=r_k e_i$. we may now estimate,  $ \int_{\Om_1^{\eps_k}} |D\tilde u_1|(x_1,z_1^k)/r_k \,dx_1$. For this we use $|D \tilde u_1(x_1,z_1^k)|=|D  u_1(x_1,z_1^k)|$ when $x_1,x_1+z_1^k\in \widetilde\Om_1$, $D \tilde u_1(x_1,z_1^k)=0$ when $x_1,x_1+z_1^k\in \om_1$ and 
$|D \tilde u_1(x_1,z_1^k)|\le \|u\|_\oo\leq 1$ when $x_1$ or $x_1+z_1^k$  belongs to $\tilde \Om_1$ but not both. In the later case we thus have $|D \tilde u_1(x_1,z_1^k)|\le  |D\un_{\om_1}(x_1,z_1^k)|$. This leads to
\begin{multline*}
 \liminf_{k\up \oo} \int_{\Om_1^{\eps_k}} \frac{|D\tilde u_1|(x_1,z_1^k)}{r_k} \,dx_1\\ 
\le \limsup_{k\up \oo} \int_{\widetilde{\Om}_1\cap\Om_1^{\eps_k}} \frac{|D u_1|(x_1,z_1^k)}{r_k} \,dx_1 
+  \limsup_{k\up \oo}\int_{\Om_1^{\eps_k}} \frac{|D\un_{\om_1}(x_1,z_1^k)|}{r_k} \,dx_1
\\
 \stackrel{\eqref{estimsumi}}{\les}  \F(u)^{\frac{1}{2}} + |e_i\cdot \nb_1 \un_{\om_1}|(\Om_1) .
 \end{multline*}
By~\eqref{lbdprime}, we have $|\nb_1 \un_{\om_1}|(\Om_1)= \h^{n_1-1}(\partial \om_1 \cap \Om_1)\les\F(u)^{\frac{1}{2}}$ and we obtain that $e_i\cdot \nb \widetilde u_1$ is a measure in $\Om_1$ which satisfies $|e_i\cdot \nb \widetilde u_1|(\Om_1)\les \F(u)^{\frac{1}{2}}$. Summing over $i$ we get
$\tilde u_1\in BV(\Om_1)$ with 
\be\label{boundtildeu1}
|\nabla \widetilde u_1|(\Om_1)\les\F(u)^{\frac{1}{2}}.
\ee
Let us note $c$ the mean value of $\tilde u_1$ (in particular $|c|\leq\|u_1\|_\oo$). Using the Sobolev injection of $BV(\Om_1)$ into $L^{n_1/(n_1-1)}(\Om_1)$, we compute
\begin{align*}
 \| u_1-c\|_{L^{\frac{n_1}{n_1-1}}(\Om_1)}&\leq  \| \tilde u_1-c\|_{L^{\frac{n_1}{n_1-1}}(\Om_1)} + \|u_1-c\|_{L^{\frac{n_1}{n_1-1}}(\om_1)} \\
 &\les |\nb_1 \widetilde u_1|(\Om_1) +  2\|u_1\|_\oo |\om_1|^{\frac{n_1-1}{n_1}} \stackrel{\eqref{boundtildeu1}\&\eqref{boundwu1u2}}\les  \F(u)^{\frac{1}{2}} + |\om_1|^{\frac{n_1-1}{n_1}}.
 \end{align*}
 By~\eqref{lbdprime2}, we have $|\om_1| \le |\Om_1|/2$ and thus by the relative isoperimetric inequality  in $\Om_1$ (see for instance~\cite[Lemma 4.5.2]{Federer}),  
 \[
 |\om_1|^{\frac{n_1-1}{n_1}}\les  \h^{n_1-1}(\partial \om_1 \cap \Om_1) \stackrel{\eqref{lbdprime}}{\les} \F(u)^{\frac{1}{2}}
 \]
so that our final estimate is  
\begin{equation*}%\label{finalu1}
\| u_1-c\|_{L^{\frac{n_1}{n_1-1}}(\Om_1)}\, \les   \F(u)^{\frac{1}{2}}. 
\end{equation*}
This concludes the proof of the proposition.
\end{proof}

\subsection*{Acknowledgments}
We thank V. Millot for very stimulating discussions during the early stages of this research. M. Goldman is partially supported by the ANR SHAPO. B. Merlet is partially supported by the INRIA team RAPSODI and the Labex CEMPI (ANR-11-LABX-0007-01).

%\cite{White1999}\cite{Fleming66}\cite{Federer}\cite{Brezis2002}

\bibliographystyle{alpha}
\bibliography{BibStripes}
\end{document}